\documentclass[12pt,reqno]{amsart}					%puts eqn numbering on right side
\input{preamble.sty}

\colorlet{red}{black}
\colorlet{blue}{black}
\title[Riemannian Submersions and Intermediate Ricci Curvature]{Do Riemannian Submersions Preserve Positive Intermediate Ricci Curvature?}

\author{Hasan M. El-Hasan \and Russell Phelan \and Frederick Wilhelm}

\address{Hasan M. El-Hasan, Department of Mathematics, University of California, Santa Barbara, Santa Barbara, CA 93106}
\email{elhasan@ucsb.edu}
\urladdr{https://www.hmelhasan.com}

\address{Russell Phelan, Department of Mathematics, University of California, Riverside, Riverside, CA 92521}
\email{rphel001@ucr.edu}

\address{Frederick Wilhelm, Department of Mathematics, University of California Riverside, Riverside, CA 92521. }
\email{fred@math.ucr.edu}
\urladdr{https://sites.google.com/site/frederickhwilhelmjr/home}

\thanks{The authors gratefully acknowledge the support of the NSF, via Award DMS 2203686.}
\date{\today}

\begin{document}
\begin{abstract}
Pro and the third author showed that there are Riemannian submersions $\pi: M \to B$ with $M$ a compact manifold with positive Ricci curvature, whose base $B$, has Ricci curvatures with both signs. Thus, Riemannian submersions need not preserve positive Ricci curvature. In this note we establish the degree to which this result extends into the setting of positive intermediate Ricci curvature. 

It is an immediate consequence of the Gray--O'Neill Equation that if $\pi: M\to B$ is a Riemannian submersion whose base is $b$-dimensional and $\mathrm{Ric}_{k}(M) >0$ for any $k \in \{ 1,2,\cdots, b-1\}$, then $ \mathrm{Ric}_{k}(B)$ is also positive. Here we show that this observation is optimal in the following strong sense:

For  $k \geq \mathrm{dim}(B)$, let $\pi: (M,g_M) \to (B,g_B)$ be a Riemannian submersion from a complete Riemannian manifold with $\mathrm{Ric}_{k}(M) >0$. We show how to perturb $g_M$ in the $C^1$-topology to produce a Riemannian submersion $\pi: (M,\tilde{g}_M) \to (B,\tilde{g}_B)$ whose total space has  $\mathrm{Ric}_{k} >0$, but whose base has Ricci curvature of both signs. 

In particular, this shows that Riemannian submersions that do not preserve positive Ricci curvature are dense in the $C^1$-topology among the complete metrics on $M$  with $\mathrm{Ric}>0$ for which a given submersion $\pi: M\to B$ is Riemannian. 
\end{abstract}

\maketitle

\section{Introduction}\label{1_introduction}
It follows from the Gray--O'Neill Equation that Riemannian submersions preserve lower bounds on sectional curvature (\cites{Gray, oneill}). In particular, they preserve positive sectional curvature. 
On the other hand,  Pro and the third author showed that they need not preserve positive Ricci curvature. More specifically they showed

\begin{custom}{Theorem}[Pro--Wilhelm, \cite{pro_wilhelm}]
	For any $C > 0$, there is a Riemannian submersion $\pi: M \to B$ for which $M$ is compact with positive Ricci curvature and $B$ has some Ricci curvatures less than $-C$.
\end{custom}

As intermediate Ricci curvatures interpolate between sectional and Ricci curvatures, it is natural to ask about the extent to which the Gray--O'Neill versus the Pro--Wilhelm results extend to the setting of intermediate Ricci curvature. Before we answer this question we recall the definition of positive intermediate Ricci curvature.

\begin{defn}
    We say that $ (M,g) $ has positive $ k^{\text{th}} $-intermediate Ricci curvature  if for all orthonormal sets $ (u, v_1,\cdots,v_k) $  we have
    \[  \sum_{i=1}^k \sec^M(u, v_i) > 0. \]
    For brevity, we will write 
    $ \ric_k^M>0 $.
    \end{defn}
    
    If $\mathcal{M}_k$ denotes the set of isometry classes of Riemannian manifolds with positive  $ \ric_k$, then
    \begin{center}
    \begin{minipage}{.6\textwidth}
        $\cal M_1$ is the class with  positive sectional curvature, \vspace{0.1in}\\
        $\cal M_{n-1}$ is the class with  positive Ricci curvature, and \vspace{0.1in}\\
        $\cal M_1 \subset\cdots\subset \cal M_k \subset \cal M_{k+1} \subset\cdots\subset \cal M_{n-1} $.
    \end{minipage}
    \end{center}

    \noindent Thus  
    the sequence of hypotheses $\{ \ric_{k}^M>0 \}_{k=1}^{n-1}$ is decreasing in strength and interpolates between the two classical hypotheses  of positive sectional and positive Ricci curvatures. 

    Let $\pi:M^n\to B^b$ be a Riemannian submersion from a compact  $n$-manifold to a compact  $b$-manifold. We will show that the Gray--O'Neill result extends to $ \ric_{k}>0 $  provided  $k \leq b-1 $, and the Pro--Wilhelm result extends  to $ \ric_{k}>0 $ provided $k \geq b$. More precisely

\begin{thmalph}\label{main-theorem}
    Let $\pi:M^n\to B^b$ be a Riemannian submersion from a compact  $n$--manifold to a compact  $b$--manifold with $b\geq 2$. Suppose that $\ric_k^M>0$ for some  $k \in \{ 1, 2, \ldots, n-1 \} $.
    \begin{enumerate}
    \item If $k\leq b-1$, then $\ric_k^B>0$.
    \item  If $k \geq b$, then $B$ need not have positive Ricci curvature.
    \end{enumerate}
\end{thmalph}

\red{We explain in Subsection \ref{Proof of Part A} how Part 1 of Theorem \ref{main-theorem} is an easy consequence of the Gray--O'Neill Equation.} 
\red{
The bulk of the rest of the paper is devoted to the proof of} the following generalization of Part 2 of {Theorem~\ref{main-theorem}}.   

\begin{thmalph}\label{trying too hard}
Let $\pi :(M,g_M) \to (B,g_{B}) $ be a Riemannian submersion from a compact $n$-manifold $M$. Suppose that $\dim (B) =b \geq2$ and for some  $k\geq b$, $\ric_k^M>0$.

Then for every $\epsilon >0$, there are Riemannian metrics $\tilde{g}_M$ and $\tilde{g}_B$ on $M$ and $B$ with the following properties. 
\begin{enumerate}
\item With respect to the new metrics  $\tilde{g}_M$ and $\tilde{g}_B$,
\begin{equation*}
\pi : (M,\tilde{g}_M) \to (B,\tilde{g}_{B})
\end{equation*}
is a Riemannian submersion.

\item  With respect to $\tilde{g}_{M}$, $\ric_{k}^M >0$.

\item $( B,\tilde{g}_{B}) $ has Ricci curvatures of both signs.

\item The $C^1$\red{-distances} of $\tilde{g}_M$ and $\tilde{g}_B$ to the original metrics satisfy,
\begin{equation*}
\norm{g_M-\tilde{g}_M}_{C^{1}}<\epsilon
\qquad\text{and}\qquad
\norm{g_B-\tilde{g}_B}_{C^{1}}<\epsilon.
\end{equation*}
\end{enumerate}
In fact, given any  $K>0$ and $p\in B$ we can choose $\tilde{g}_{M}$ to have all of the properties above and,  additionally, satisfy
\begin{eqnarray}\label{real neg}
\sec (B,\tilde{g}_{B})|_p &< & -K ,  \; \; \text{and} \\
  \ric_k(M,\tilde{g}_{M}) &\geq & \min{ \ric_k (M,g_M) } - \epsilon. \label{almost the same}
\end{eqnarray}

\end{thmalph}

\red{When $k=n-1$, $M$ has positive Ricci curvature. So Theorem \ref{trying too hard} implies that   among metrics with positive Ricci curvature for which a given submersion is Riemannian, those that do not project to positive Ricci curvature always exist, and in fact are dense in the $C^1$-topology. We record this observation in the following corollary.}

\begin{coralph}\label{PW pheonenon is C-1 dense} \red{Let $\pi :M \to B $ be a  submersion from a compact $n$-manifold. Let $\mathcal{R}_{\pi}^+$ be the set of Riemannian metrics $g$ on  $M$ with respect to which $\pi  :M \to B$ is  Riemannian and $ \ric(M,g) >0$. Then the subset of $\mathcal{R}_{\pi}^+$ for which the induced metric on $B$ has Ricci curvature of both signs is $C^1$-dense in $\mathcal{R}_{\pi}^+$.}
\end{coralph}
 \red{
In the event that the submersion $\pi :(M,g_{M})\to (B,g_{B})$ of Theorem \ref{trying too hard} is the quotient map of a free isometric action, our metric transformation, 
  $g_M \rightsquigarrow \tilde{g}_M $, preserves this property (see Section \ref{4_definition_of_the_metric}). More formally, the following is a corollary of the proof of Theorem \ref{trying too hard}.}

\begin{coralph}\label{grp action cor}\red{
Let $\pi :(M,g_{M})\to (B,g_{B})$ be as in the hypothesis of
Theorem \ref{trying too hard}. If, in addition, $\pi $ is the quotient map of a free isometric
group action 
\begin{equation}\label{iso}
G\times M\to M,
\end{equation}
the metric $\tilde{g}_{M}$ in the conclusion of Theorem \ref{trying too hard}, can be chosen so
that (\ref{iso}) is also isometric with respect to $\tilde{g}_{M}.$}
\end{coralph}

\begin{remark-nonumb}

In \cite{galaz-garcia}, Galaz-Garc\'{i}a, Kell, Mondino, and Sosa, show that four of the synthetic analogs of a lower bound on Ricci curvature {\em are} preserved under the quotient maps of  isometric group actions. In light of this and the results of this paper, it is natural to ask how one should define the synthetic analog of intermediate Ricci curvature bounded from below.  
\end{remark-nonumb}

\red{In Section \ref{2_notation} we establish notation and review the definition of the $C^1$-topology on the space of $(0,2)$ tensors.} In particular, because it is convenient for the proof of Theorem \ref{trying too hard}, we define the $C^1$-norm of a $(0,2)$-tensor relative to fixed background metric. We then compare this definition to the standard one that is given in terms of a fixed atlas. 

In Section \ref{3_construction_of_functions}, we show how to construct $C^1$\red{-small} functions while controlling  their Hessians. The main result of the section, Lemma \ref{existence lemma omega}, will be used later in the paper to construct functions, $\omega_h,\omega_v :B \to \R$, which we use to define the metrics $\tilde{g}_M$ and $\tilde{g}_B$ that prove Theorem \ref{trying too hard}. The way that  $\omega_h$ and $\omega_v$ determine  $\tilde{g}_M$ and $\tilde{g}_B$ is specified in Section \ref{4_definition_of_the_metric}.

In Section \ref{5_verification_on_the_base}, we verify Inequality (\ref{real neg}). Since Inequality (\ref{real neg}) asserts that {\em all} sectional curvatures of  $B$ at $p$ are  smaller than $-K$, a  fortiori, 
$$
\ric (B,\tilde{g}_{B})|_p < -K (b-1).
$$
This together with Theorem 2 of \cite{pro_wilhelm} implies that $\tilde{g}_B$ has Ricci curvature of both signs, provided $ \ric_k(M,\tilde{g}_{M}) >0$.
Section \ref{6_methods_to_verify_PIRC} establishes a  method to verify $\ric^M_k>0$, which is based in part, on a lemma of Reiser and Wraith from \cite{reiser_wraith}. Finally, Section \ref{7_analysis_on_the_total_space} completes the proof of Theorem \ref{trying too hard} by verifying positive intermediate Ricci curvature on the total space\red{, and in the final paragraph of the paper we explain how Corollary \ref{grp action cor} follows from the proof of  Theorem \ref{trying too hard}.}

\begin{remark-nonumb}
    An earlier draft of this paper appeared as a chapter in the first author's doctoral dissertation \cite{elhasan}.
\end{remark-nonumb}

{\color{red}\begin{acknowledgment}
    We are grateful to the referee for their careful reading and thoughtful critiques of the original submission. We are especially grateful for the recommendation that Corollary \ref{PW pheonenon is C-1 dense} be formally stated and for the question that lead us to realize that Corollary \ref{grp action cor} is true and a worthy addition to the paper. 
\end{acknowledgment}}
%%%%%%%%%%%%%%%%%%%%%%%%%%%%%%%%%%%%%%%%%%%%%%%%%%%
%%%%%%%%%%%%%%%%%%%%%%%%%%%%%%%%%%%%%%%%%%%%%%%%%%%
%%%%%%%%%%%%%%%%%%%%%%%%%%%%%%%%%%%%%%%%%%%%%%%%%%%
\section{Notation and Preliminary Results}\label{2_notation}
We denote the Riemannian connection, curvature tensor, and curvature operator by $\nabla $, $R$, and $\cal R $ and the vertical and horizontal distributions of a submersion $\pi:(M^n,g_M)\to(B^b,g_B)$ by $\cal V$ and $\cal H$, respectively. The horizontal and vertical components of a vector $E$ are denoted by $E^{\mathbf{h}}$ and $E^{\mathbf{v}}$, respectively. We will write $X$, $Y$, and $Z$ to represent horizontal vector fields, and $T$, $U$, and $V$ to represent vertical vector fields.

We will rely on several curvature calculations from Gromoll and Walschap's book (\cite{gromoll_walschap}), so to facilitate the verification of our assertions, we adopt the notation system in \cite{gromoll_walschap} for the fundamental tensors of a Riemannian submersion, rather than those of Gray and O'Neill (\cite{Gray}, \cite{oneill}). For the reader's convenience, we review these here. 

We define the tensor $A:\cal H\times\cal H\to\cal V$ by
\begin{equation}\label{A-tensor}
A_X Y = \frac{1}{2}[X,Y]^\v,
\end{equation}
and its adjoint $A^*:\cal H\times\cal V\to\cal H$ is denoted $A^*_X V$. We also define the shape operator  $S:\cal H\times\cal V\to\cal V$ and second fundamental form of the fibers $\sigma:\cal V\times\cal V \to\cal H$ 
by
\begin{eqnarray}\label{S-tensor}
S_X U & = & -(\nabla_U X)^\v \; \; \text{and} \\
 \sigma(U,V) &  = & (\nabla_U V)^\h.   \label{sigma-tensor}
\end{eqnarray}
%%%%%%%%%%%%%%%%%%%%%%%%%%%%
The operation $\circ$ denotes the Kulkarni--Nomizu product of two $(0,2)$-tensors, which is given by
\begin{align}\label{Kulkarni Nomizu}
    \begin{split}
    \alpha\circ\beta(v_1,v_2,v_3,v_4) 
    =\;&
    \frac{1}{2}\left( \alpha(v_1, v_4)\beta(v_2, v_3) + \alpha(v_2, v_3)\beta(v_1, v_4) \right)\\
    &-
    \frac{1}{2}\left( \alpha(v_1, v_3)\beta(v_2, v_4) + \alpha(v_2, v_4)\beta(v_1, v_3) \right). \;\; 
    \end{split}
\end{align}
(See Exercise 3.4.23 in \cite{petersen_book} \red{for more information about the Kulkarni--Nomizu product}.)

\subsection{\red{Decreasing Horizontal Intermediate Ricci curvatures}}\label{Proof of Part A} 
\red{Next we briefly explain how Part 1 of Theorem \ref{main-theorem} follows from the Gray–O’Neill Equation. }

\red{
Suppose that $\ric_{k}^M >0$ for some $k\in \left\{1,2,\ldots b-1\right\},$ and let $\left\{ u,v_{1},\ldots,v_{k}\right\} \subset TB $ be any  orthonormal set.
 Let $\left\{ \tilde{u},\tilde{v}_{1},\ldots ,\tilde{v}_{k}\right\} $ be the horizontal lift of $\left\{ u,v_{1},\ldots ,v_{k}\right\} $ to $TM.$ Then 
\begin{eqnarray*}
\sum_{i=1}^{k}\sec^{B}\left(u,v_{i}\right) &=&\sum_{i=1}^{k}\left( \sec^{M}\left( \tilde{u},\tilde{v}_{i}\right) +3\left\vert A_{\tilde{u}}\tilde{v}_{i}\right\vert ^{2}\right)  \\
&\geq &\sum_{i=1}^{k}\sec^{M}\left( \tilde{u},\tilde{v}_{i}\right) 
\\
&>&0,
\end{eqnarray*}
where $A$ is the Gray--O'Neill $A$-tensor. Thus $\ric_k^B>0$ as claimed in Part 1 of Theorem \ref{main-theorem}.}

\subsection{\texorpdfstring{$C^{1}$-Topology on the space of $(0,2)$--tensors}{C¹-Topology on the space of (0,2)--tensors}}
%%\red{We recommend the user refer to Hirsch's textbook (\cite{Hirsch_1997}) for more information about the $C^k$-norm of a smooth %%%function on a manifold.}

Fix a compact smooth manifold $M${\color{red}, and a $(0,2)$-tensor  $h$ on  M. The typical definition of the $C^0$- and $C^1$-norms of $h$ depend on a choice of a fixed finite atlas $\mathcal{A}$. For the readers convenience, we repeat them formally here.}  
{\color{blue} We refer the reader  to Chapter 2 of Hirsch's textbook (\cite{Hirsch_1997}) for the corresponding definitions of the $C^0$-norm and $C^1$-norms of smooth real valued functions on $M$.}

\begin{defn}
Let $M$ be a compact, smooth $n$-manifold with fixed finite atlas \red{$\mathcal{A}:=\{(U,\phi _{U})\}$}. Assume that for all $(U,\phi _{U}) \in \mathcal{A}$, there is a neighborhood $V$ of the closure of $U$ so that the
\begin{equation}\label{extended chart}
  \text{coordinate chart} \;  \phi \; \text{extends to} \;  V.
\end{equation}
Given $(U,\phi _{U})\in \mathcal{A}$ and a $(0,2)$-tensor $h$, write 
\begin{equation*}
h_{i,j}:=h(E_{i},E_{j}),
\end{equation*}%
where $E_{i},E_{j}$, are coordinate fields of  $(U,\phi _{U})$.
%%%%%%%%%%%%%%%%%%%%%%%%%%%%%%%%%%%%
The $C^{0}$-norm of $h$ with respect to $\mathcal{A}$ is then 
\begin{equation*}
\norm{h}_{C^{0},\mathcal{A}}:=\max h(E_{i},E_{j}),
\end{equation*}
where the maximum is taken  over all of the coordinate fields $E_{i},E_{j}$ of all of the coordinate charts of $\mathcal{A}.$ Similarly, the $C^{1}$-norm of a $(0,2)$-tensor $h$ with respect to $\mathcal{A}$ is  
\begin{equation*}
\left\Vert h\right\Vert _{C^{1},\mathcal{A}}:=\max \{|\partial
_{k}h_{i,j}|,\left\Vert h\right\Vert _{C^{0},\mathcal{A}}\}
\end{equation*}
where the maximum is taken over all of the coordinate fields $E_{i},E_{j}$\red{,} $E_{k}$ of all of the coordinate charts of $\mathcal{A}.$ 
\end{defn}

For us, it will be simpler and more convenient to define these notions in terms of a fixed background Riemannian metric $g_0$. 
 
\begin{defn}
Fix a compact Riemannian manifold $(M,g_0)$. 
    Given $\epsilon >0$ and a $(0,2)$-tensor $h$ we say that 
    \begin{equation*}
    \norm{h}_{C^{0},g_0}<\epsilon,
    \end{equation*}
    provided for all $u\in TM$,
    \begin{equation*}
     \abs{h(u,u)}\leq \epsilon g_0(u,u).
    \end{equation*}
\end{defn}

\begin{defn}
    Let $(M,g_0)$ be a compact Riemannian manifold. Given $\epsilon >0$ and a $(0,2)$-tensor $h$, we say that 
    \begin{equation*}
    \norm{h}_{C^{1},g_0}<\epsilon ,
    \end{equation*}
    provided
    \begin{equation*}
    \norm{h}_{C^{0},g_0}<\epsilon ,
    \end{equation*}
    and for all $u,w\in TM,$
    \begin{equation*}
    \abs{\left( \nabla_{w}^{g_0} h\right) \left( u,u\right)} \leq \epsilon
    g_0(u,u)\norm{w}_{g_0},
    \end{equation*}
    where $\nabla ^{g_0}$ denotes covariant derivative with respect to $g_0$.
\end{defn}
All four notions of norms then yield a metric space structure in the usual way, that is, 
\begin{equation*}
    \dist(h, \tilde{h}) := \| h- \tilde{h} \|,
\end{equation*}
where $\| h- \tilde{h} \|$ is any of the four notions of norms defined above. 

Finally, we record the fact for  $p\in \{0,1\}$, the two  $C^p$-norms yield the same topology on the set of $(0,2)$-tensors.

\begin{prop}\label{two norms prop}
    Let $M$ be a compact, smooth $n$-manifold with a fixed Riemannian metric $g_0$ and a fixed finite atlas 
    $\mathcal{A}:=\{(U,\phi _{U})\}$ that satisfies \eqref{extended chart}. Given $\epsilon >0$ there are $\delta ,\tilde{\delta}>0$ so that,  for $p\in \{0,1\}$ and any $(0,2)$-tensor $h$, if 
   \begin{equation} \label{coord implies invar}
    \norm{ h }_{C^{p},\mathcal{A} }  < \delta,    \text{ then }
    \norm{ h  }_{C^{p},g_0} < \epsilon,
    \end{equation}
    and if 
    \begin{equation} \label{inv implies coord}
    \norm{ h }_{C^{p},g_0}  < \tilde{\delta},    \text{ then }
    \norm{ h  }_{C^{p},\mathcal{A}} < \epsilon .
    \end{equation}
\end{prop}

\begin{proof}
For  $u \in TM$ we set, 
\[\norm{u}^{\max}_{\eu} := \max \left\{ \norm{d \phi _{O} (u )}_{\eu} \; | \; u \in TO, \; \text{where} \; (O,\phi _{O})\in \mathcal{A} \right\},
\]
that is $\norm{u}^{\max}_{\eu}$ is the largest of all of the possible euclidean lengths of $d \phi _{O} (u)$ where $(O,\phi _{O})$ is any of the  finite number of charts of  $\mathcal{A}$ that contain the foot point of $u$.

Since $\mathcal{A}$ satisfies \eqref{extended chart} and  there are only a finite number of charts in $\mathcal{A}$, there is a  $K>1$ so that for all $u\in TM$,    
\begin{equation}\label{compactness}
    \frac{\norm{u}^{\max}_{\eu}}{K} \leq   \norm{u}_{g_0}   \leq K \norm{u}^{\max}_{\eu},
\end{equation}
and if  $U,W$ are any coordinate fields of any chart $(Q,\mu)$ of $\mathcal{A}$, then 
\begin{equation}\label{compactness--3}
   \max\left\{ \ \norm{\nabla_{W}^{g_0} U}_{\eu}^{\max}  \right\} \leq K.
\end{equation}
The versions of  (\ref{coord implies invar}) and  (\ref{inv implies coord}) for the  $C^0$-topology follow from \eqref{compactness}.

If $\{ E_i \}$ are the coordinate fields of the chart $(Q,\mu)$ and $\langle \cdot  ,\cdot \rangle_Q$ is the euclidean dot product in $\mu(Q)$, then 
\begin{equation*}\label{covar norm}
 \nabla_{W}^{g_0} U = \sum_{i=1}^n \langle \nabla_{W}^{g_0} U , E_i\rangle_Q E_i.
\end{equation*}
%%%%%%%%%%%%%%%%%%
Thus for any other coordinate field  $V$ of $(Q,\mu)$,
\begin{eqnarray}\label{compactness--4}
\abs{h  (\nabla_{W}^{g_0} U,V)} & \leq& \abs{h  \protect\red{\left(\sum_{i=1}^n \langle \nabla_{W}^{g_0} U , E_i\rangle_Q E_i,V\protect\right)}} \nonumber  \\
&\leq& \sum_{i=1}^n \abs{\inn{\nabla_W U, E_i}_Qh(E_i,V)}  \nonumber \\
&\leq& Kn \norm{h} _{C^{0},\mathcal{A}},
\end{eqnarray}
by \eqref{compactness--3}. 
%%Since there are are only a finite number of possibilities for coordinate fields for our atlas,

This gives us
\begin{eqnarray*}
    \abs{ ( \nabla_{W}^{g_0} h ) (U,V)} 
    &\leq&
    \abs{D_{W}h(U,V)} 
    + \abs{h(\nabla_{W}^{g_0} U,V)}  + \abs{h(U, \nabla_{W}^{g_0}V)}\\
    &\leq&
    \norm{h } _{C^{1}, \mathcal{A}} +
    2K n  \norm{h} _{C^{0},\mathcal{A}}  \; \; \text{by \eqref{compactness--4}} \\
    &\leq& 
    (1+2K n )  \norm{h} _{C^{1}, \mathcal{A}}.
\end{eqnarray*}
The version of (\ref{coord implies invar}) where $p=1$ follows by combining this with \eqref{compactness} and the version of (\ref{coord implies invar}) for $p=0$.

In the event that $\norm{h}_{C^{1}, \mathcal{A} } = \norm{h}_{C^{0}, \mathcal{A} } $,  the version of
(\ref{inv implies coord}) for $p=1$  follows from the  version for $p=0$. Otherwise, we let $U,V,$ and $W$ be coordinate fields so that at a point  $p$,
\[\abs{D_{W}h(U,V )} = \norm{h}_{C^{1}, \mathcal{A} }.\]
\red{By combining} this with \eqref{compactness}  and \eqref{compactness--4} we see that
\begin{eqnarray*}
    \norm{h}_{C^{1}, \mathcal{A} }
    & = & 
    \abs{D_{W} h (U,V) )}  \\ 
    &\leq&  
    \abs{\left(\nabla_{W}^{g_0} h\right) (U, V) }
    + \abs{h  (\nabla_{W}^{g_0} U,V)}   + \abs{h  ( U, \nabla_{W}^{g_0}V)}\\
    &\leq&  
    K^3\norm{h } _{C^{1},g_0}  
    + 2 Kn \norm{h } _{C^{0},\mathcal{A}}.
\end{eqnarray*}
%%%%%%%%%%%%%%%%%%%%
The version of (\ref{inv implies coord}) where $p=1$ follows this and the version of (\ref{inv implies coord}) where $p=0$.
\end{proof}

We use the notation $O\left(t \right) $ to denote a number, vector, or tensor whose norm is no larger
than $Ct$, where $C$ is a positive constant that only depends on $g_0$.

We write $\inj_M$  for the injectivity radius of  $M$.

%%%%%%%%%%%%%%%%%%%%%%%%%%%%%%%%%%%%%%%%%%%%%%%%%%%
%%%%%%%%%%%%%%%%%%%%%%%%%%%%%%%%%%%%%%%%%%%%%%%%%%%
%%%%%%%%%%%%%%%%%%%%%%%%%%%%%%%%%%%%%%%%%%%%%%%%%%%

\section{Construction of Functions}\label{3_construction_of_functions}
To create the  metric $\tilde{g}_M$ that proves Theorem \ref{trying too hard},  we will construct two functions using the following lemma. 

\begin{lemma}\label{existence lemma omega}
    Let $M$ be a compact Riemannian $n$-manifold. Given any $p\in M$, $C\in\R\setminus[-1,1]$,  $\epsilon,\eta \in \left(0, \frac{1}{2} \min \{ 1,  \ \inj_M  \}\right)$, and any   $\tau\in (0,\frac{\epsilon \eta }{2|C|})$ there is a smooth function 
    \[\omega : M\to\R\]
    so that
    \begin{enumerate}
        \item $\left.(\grad \omega) \right|_p = 0$,
        \item $\norm{\omega}_{C^1} < \epsilon$,  and
        \item $\supp(\omega) \subset B(p,2\eta)$.
    \end{enumerate}
    Moreover, if  $C>0$, we can choose $\omega$ so that for all unit vectors $Z \in TM$,
    %%%%%%%%%%
    \begin{equation} \label{Hessian inequal +}
     - \epsilon \leq \hess_\omega(Z,Z) \leq 3C, 
    \end{equation}
    and for all unit vectors $Z \in TM|_{B(p,\tau)}$,
    \begin{equation} \label{Hessian inequal +b}
        C \leq \hess_\omega(Z,Z).
    \end{equation}
    If  $C<0$, we can choose $\omega$ so that for all unit vectors $Z \in TM$,
    %%%%%%%%%%
    \begin{equation} \label{Hessian inequal -}
   3C \leq \hess_\omega(Z,Z) \leq \epsilon.
    \end{equation}
    and for all unit vectors $Z \in TM|_{B(p,\tau)}$,
     \begin{equation} \label{Hessian inequal -b}
   \hess_\omega(Z,Z) \leq C.
    \end{equation}
\end{lemma}
Before proving Lemma \ref{existence lemma omega} we establish two preliminary single variable calculus results.
\begin{lemma}\label{gluing_lemma}
\red{For every $\epsilon, \eta >0$, $K>1$, and $\delta \in (0, \frac{ \epsilon}{ 3K } )$, 
 let $\lambda :\mathbb{R}\to \left[ 0,1\right] $ be a $C^2$-smooth function satisfying}
\begin{eqnarray}
\lambda |_{\R\setminus[-2\eta,2\eta]} &\equiv& 0, \nonumber \\
\lambda |_{[-\eta,\eta] } &\equiv& 1, \quad \text{and} \nonumber\\
\| \lambda \|_{C^2} &<& K. \label{C 2 norm}
\end{eqnarray}
Let $f,g:\left[ -3\eta,3\eta\right] \to \R$ be $C^2$ and set
\begin{equation}\label{weighted aver}
\phi :=\lambda f + (1-\lambda)g.
\end{equation}
If
\begin{equation}\label{C^1 small}
   \norm{f-g}_{C^{1}}<\delta, 
\end{equation}
then
\begin{eqnarray}\label{da glue}
\norm{(f-\phi) }_{C^{1}}<\epsilon, \; \;
\norm{(g-\phi) }_{C^{1}}<\epsilon, \; \; \text{and} \; \; \norm{( \phi'' - ( \lambda f'' + (1-\lambda)g'')) }_{C^{0}} <\epsilon.
\end{eqnarray}
\end{lemma}
\begin{proof}
\red{Hypotheses} (\ref{C 2 norm}), (\ref{weighted aver}), and  (\ref{C^1 small}) together give us
\begin{eqnarray}\label{deriv versus aver}
| \phi' - ( \lambda f' + (1-\lambda)g') | & = & |\lambda'(f-g)|   \\
        & < & K  \delta   \nonumber\\
        & < &  \epsilon, \qquad \protect\red{\text{since } \delta \in \left(0, \frac{ \epsilon}{ 3K } \right),} \nonumber
\end{eqnarray}
\blue{and}
\begin{eqnarray*}\label{2nd deriv versus aver}
| \phi'' - ( \lambda f'' + (1-\lambda)g'') | & = & |2 \lambda'(f'-g') + \lambda''(f-g) |   \\
        & < &  3 K  \delta    \\
        & < &  \epsilon, \qquad \protect\red{\text{since } \delta \in \left(0, \frac{ \epsilon}{ 3K } \right),} 
\end{eqnarray*}
%%%%%%%%%%%%%%%%%%%%%%%
so the last inequality in \eqref{da glue} holds.

Equation (\ref{weighted aver}) and \eqref{deriv versus aver} tell us that the values of $\phi$ and $\phi'$ never deviate further than $\epsilon$ from a weighted average of $f$ and $g$ and  $f'$ and $g'$, respectively. This together with $\norm{f-g}_{C^{1}}<\delta <\epsilon$, yields the first two inequalities in (\ref{da glue}).
\end{proof}

\begin{lemma}\label{existence_of_special_function_lemma}
    Given  $\epsilon,\eta \in(0,1)$, $C\in\R\setminus[-1,1]$, and any  $\tau\in (0,\frac{\epsilon \eta}{2|C|})$, there is a function $\phi:\R\to\R$   with the following properties:
    {\color{red}
    \begin{enumerate}
    \item \[\phi(t) = \frac{C}{2} t^2 \; \; \text{for}  \;\; t \in [-\tau ,\tau ],\]
    \item \begin{equation}\label{existence_C1}
        \norm{\phi}_{C^1}<\epsilon \text{ on } \R,
    \end{equation}
    \item \[\phi\equiv0 \text{ on } \R\setminus[-2\eta,2\eta],\]
    \item 
    \begin{equation}\label{existence_C2}
        \norm{\phi}_{C^2}<\epsilon 
        \text{ on } \R\setminus[-\eta,\eta].
    \end{equation}
    \end{enumerate}
    }
    Moreover if $C>0$, we can choose  $\phi$ so that  
    \begin{equation}\label{deriv near 0}
        \phi'|_{[0,\eta]}\geq 0 \; \; \text{and} \;\; \phi''\in(-\epsilon,C],
    \end{equation}
     and if $C<0$, we can choose  $\phi$ so that
     \begin{equation}\label{deriv near 0 cneg}
        \phi'|_{[0,\eta]}\leq 0 \; \; \text{and} \;\; \phi''\in[C,\epsilon).
    \end{equation}
   
\end{lemma}
\begin{proof} \red{We first} consider the case where $C>0$. \red{The result when $C<0$ follows readily from the result for positive $C$, we explain this in more detail at the end of the proof.}  

    For some 
    \begin{equation}\label{choice of nu}
       \nu\in\left(0,\frac{\epsilon \eta }{2C}\right) 
    \end{equation}
    and $\tau \in (0, \nu)$, let $h_\nu: \R \to [0,\infty)$ satisfy
    \begin{equation}\label{existence h nu}
        h_\nu(t) = C  \quad \text{for} \quad t\in [-\tau, \tau], 
    \end{equation}
    \begin{equation}\label{existence h nu prime}
      h_\nu|_{(-\nu,\nu)} \geq 0,\quad  h'_\nu|_{(-\nu,0]} \geq 0, \quad h'_\nu|_{[0,\nu)} \leq 0,  \quad \text{and} \quad h_\nu|_{ \R\setminus[-\nu, \nu] }\equiv 0.
    \end{equation}
    %%%%%%%%%%%%%%%
We set
    \[
    f(x) := \int_0^x \int_0^t h_\nu(s)\odif{s}\odif{t},
    \]
  and it follows that
    %%%%%%%%%%%%%%%%%%%%%%%%%
    \begin{equation*}
 f'(t) := \int_0^t h_\nu(s)\odif{s}  \; \; \text{and} \;\;  f''(s)= h_\nu(s).
    \end{equation*}
      %%%%%%%%%%%%%%%%%%%%%%%%%
  In particular, 
    \begin{eqnarray}\label{f on 0 eta}
        f'(t)\geq 0 \quad &  \text{for all}&  t\geq 0, \\
    f(t)= \frac{C}{2} t^2  & \text{for all} &  t\in [-\tau,\tau], \label{f prime prime is C on -tau tau}
    \end{eqnarray}
    and
    \begin{eqnarray*}
        \norm{f|_{[-2\eta ,2\eta]}}_{C^1} 
    &\leq& \sup_{x\in[-2\eta, 2\eta]} \abs{\int_0^x \int_0^t h_\nu(s) \odif{s}\odif{t}}
        + \sup_{t\in[-2\eta, 2\eta]}\abs{\int_0^t h_\nu(s) \odif{s}} \\
    &\leq& C\nu\eta + C\nu.
    \end{eqnarray*}
    Since 
    \[
    f''|_{\R \setminus [-\eta,\eta]} = h_\nu|_{\R\setminus[-\eta,\eta]} \equiv 0,
    \]
    we also have that $\norm{f}_{C^2}=\norm{f}_{C^1}$ on $\R \setminus [-\eta,\eta]$, so
   \begin{eqnarray*}
        \norm{f|_{[-2\eta,2\eta]\setminus [-\eta,\eta]}}_{C^2} &\leq &C\nu\eta+C\nu \\
                       &< &   \epsilon \; \; \text{by \eqref{choice of nu}}.
   \end{eqnarray*}

    To construct $\phi$, we  apply Lemma \ref{gluing_lemma} with $f$ as above and $g\equiv 0$. 
     Note that since 
    $\phi|_{[0,\eta]}= f|_{[0,\eta]}$ it follows from (\ref{f prime prime is C on -tau tau}) that 
    $$
   \phi(t) = \frac{C}{2} t^2  \;\; \text{for}\; \;  t\in [-\tau,\tau].
    $$
Similarly, $\phi|_{[0,\eta]}= f|_{[0,\eta]}$ together with (\ref{f on 0 eta}) gives us that    
    $\phi'|_{[0,\eta]}\geq 0$, as claimed in (\ref{deriv near 0}). 
   %%% Note that there is a $K>0$ that depends only on $\eta$ so that $\norm{\lambda}_{C^2}<K$. 
   
  Combining Lemma \ref{gluing_lemma} with the fact that  the second derivatives of both $f$ and the zero function are zero on $\R \backslash [-\eta , \eta]$,  we see that 
    \begin{equation}\label{close to averg}
          \norm{ (\phi-f)|_{[-2\eta , 2\eta] \backslash [-\eta , \eta]}}_{C^2} < \epsilon  \; \; \text{and} \; \; \norm{\phi|_{\R \backslash [-\eta , \eta]}}_{C^2} < \epsilon.
    \end{equation}
    Thus (\ref{existence_C2}) holds, and we have that 
    $\phi''>-\epsilon$ on $\R \backslash [-\eta , \eta]$. Since $\phi = g\equiv0$ on $\R\setminus[-2\eta,2\eta]$, we have 
    $$
    \phi \equiv 0 \; \; \text{on} \; \; \R\setminus[-2\eta,2\eta].
    $$
    Since $\phi=f$ on $[-\eta , \eta]$, we have that $\phi''|_{[-\tau,\tau]} =f''|_{[-\tau,\tau]}=C$, and since $f'' = h$ is nonnegative on $[-\eta , \eta]$ and has a maximum at $0$ (see (\ref{existence h nu}) and (\ref{existence h nu prime})), we have that 
    \[\phi''\in (-\epsilon,C],\]
    which completes the proof of (\ref{deriv near 0}).
    
    Finally,
    \begin{eqnarray*}
     \norm{\phi|_{[-\eta , \eta]}}_{C^1} &=& \norm{f_{[-\eta , \eta]}}_{C^1} \\
                       &< &C\nu\eta + C \nu \\
                        &< & 2C\nu, \; \; \text{since $\eta <1$} \\
                    &<& \epsilon,
     \end{eqnarray*}
  since  $\nu < \frac{\epsilon }{2C}$, which together with (\ref{existence_C2}) completes the proof of (\ref{existence_C1}).

\red{
  For the case when $C<0$, we apply the result that we just proved to $|C|$ and call the resulting function  $\psi$. The desired function is then $\varphi := -\psi$. }
\end{proof}

%%%%%%%%%%%%%%%%%%%%%%%%%%%%%%%%
\begin{proof}[Proof of Lemma \ref{existence lemma omega}]
Given $p, C,\epsilon, \eta$, and $\tau$ as in Lemma \ref{existence lemma omega}, we first apply Lemma \ref{existence_of_special_function_lemma} with $\epsilon$, $C$, $\eta$, and $\tau $ transformed as indicated by the following table:

\begin{equation}\label{da table}
    \begin{tabular}{|c|c|c|}
    \hline
    Lemma \ref{existence lemma omega} Quantity & $\longrightarrow $ & Lemma \ref{existence_of_special_function_lemma} Quantity
    \\ \hline
    $\epsilon $ & $\longmapsto $ & $\epsilon \eta ^{3}$ \\ \hline
    $C$ & $\longmapsto $ & $2C$ \\ \hline
    $\eta $ & $\longmapsto $ & $\eta $ \\ \hline
    $\tau $ & $\longmapsto $ & $\tau $ \\ \hline
    \end{tabular}%
    \end{equation}

Noting that if the conclusion of Lemma \ref{existence lemma omega} holds for a given $\eta \in \left(0, \frac{1}{2} \min \{ 1,  \ \inj_M  \}\right)$, then it also holds for all $\tilde{\eta} \in \left(\eta, \frac{1}{2} \min \{ 1,  \ \inj_M  \}\right) $, we will impose one further restriction on the  $\eta$ of Lemma \ref{existence_of_special_function_lemma} toward the end of the proof. 
    
Let $\phi:\R\to\R$ be the function resulting from our application of Lemma \ref{existence lemma omega}, and set
    \begin{eqnarray*}
    \omega(x) &:=& \phi \circ \left(\dist_p(x)\right) \; \; \text{and} \\
           X  &:=&  \nabla\dist_p.
    \end{eqnarray*}
Since $\phi(t) = C t^2 $ on $ [-\tau ,\tau ]$, we have that the restriction of  $\omega$ to $B(x,\tau)$ is 
    \begin{equation}\label{existence lemma omega definition}
        \omega(\cdot) = C \left(\dist_p(\cdot)\right)^2,
    \end{equation}
so $\omega$ is smooth on $B(x,\tau)$. Since $\dist_p$ is smooth on $B(p,\inj_M) \backslash \{ 0\}$, and $\supp(\phi)\subset [-2\eta,2\eta]$, and $2\eta\in(0,\inj_M)$, it follows that $\omega$ is smooth and supported in $B(p,2\eta)$.

Since $\norm{\phi}_{C^0}<  \epsilon \eta^3$, we have that $\norm{\omega}_{C^0}< \epsilon \eta^3 <\epsilon$. Moreover,
\[
(\grad \omega)|_x = \phi'\left(\dist_p (x)\right)X,
\]
and since $\abs{\phi'\left(\dist_p (x)\right)}<  \epsilon \eta^3 $, it follows that $\abs{(\grad \omega)|_x} <  \epsilon \eta^3$, and hence $\norm{\omega}_{C^1} <  \epsilon \eta^3 < \epsilon$. 

Since $\phi'(0)=0$, it follows that $(\grad \omega)|_p = 0$. 

It remains to prove Inequalities (\ref{Hessian inequal +})--(\ref{Hessian inequal -b}). Apart from mechanical changes the proofs when $C>0$ and $C<0$ are identical, so we only detail the proofs when $C>0$, that is of (\ref{Hessian inequal +}) and (\ref{Hessian inequal +b}). 

Since
\begin{eqnarray}\label{Hessian on XX}
    \hess_\omega(X,X)|_x = \phi''\left(\dist_p (x)\right) \; \; \text{and}\; \; 
\phi''\in (-\epsilon \eta^3, 2C],
\end{eqnarray}

we have that 
\[-\epsilon \eta^3 \leq \hess_\omega(X,X)\leq 2C.\]

    On the other hand, if $Y\in T_xM$ is a unit vector perpendicular to $X$, then we have $\nabla_Y \left(\phi' \circ \dist_p(\cdot)\right) = 0$, so 
    %%%%%%%%%%%%%%%%%%
  \begin{eqnarray}\label{Hess Y}
       \nabla_Y \grad \omega &=& \phi' ( \dist_p(\cdot) ) \ \nabla_Y X, \; \; \text{and} \\
                        \hess_\omega(Y,X)&=&   \phi' ( \dist_p(\cdot) ) \ g(  \nabla_Y X , X)  \nonumber\\
                                          &=&   0, \nonumber
  \end{eqnarray}
since $\| X \| \equiv 1$.
    %%%%%%%%%%%%%%
Equation (\ref{Hess Y}) also gives us
    \begin{eqnarray}\label{tang Hessian}
        \hess_\omega(Y,Y)|_x &=& g\left( \nabla_Y \grad \omega, Y \right)|_x \nonumber \\
        &=& \phi'\left(\dist_p (x)\right) g\left( \nabla_Y X , Y\right)|_x  \nonumber  \\
        &=& \phi'\left(\dist_p (x)\right) \hess_{\dist_p}(Y,Y)|_x   \nonumber \\
        &=& \phi'\left(\dist_p (x)\right) \left( \frac{1}{\dist_p (x)} + O\left( \dist_p (x) \right) \right),
    \end{eqnarray}
      %%%%%%%%%%%%%%
    where for the last equality we are using Equation 2.7.1 in \cite{searle_wilhelm}. The final restriction that we impose on the quantity $\eta$ of Lemma \ref{existence_of_special_function_lemma} will be to choose it small enough so that on $B(p,2\eta)$ the error term on the right hand side of $\eqref{tang Hessian}$ satisfies
    \begin{equation}\label{bnd on HWTL}
    \abs{O\left( \dist_p (x) \right)}< \sqrt{\eta} \; \; \text{on $B(p,2\eta)$}.
    \end{equation}

    Since $\phi'(0) = 0$ and $\phi''(t)\leq 2C$, we have $\phi'(t)\leq 2Ct$ for $t>0$. This together with (\ref{tang Hessian}), \eqref{bnd on HWTL}, and $\|\phi\|_{C^1}<\epsilon \eta^3$ gives us that for $x\in \supp(\omega)$,
    \begin{eqnarray*}
    \hess_\omega(Y,Y) &\leq&
    2C\dist_p (x) \frac{1}{\dist_p (x)} + \epsilon\eta^3 \cdot O\left( \dist_p (x) \right) \\
     &\leq& 2C +  \epsilon\eta^3 \\
    &\leq& 3C, \; \; \;  \text{since $\epsilon, \eta \in (0, \frac{1}{2}  )$ and $1< C$.}
    \end{eqnarray*}
Since   $\phi'|_{[0,\eta]} \geq 0$ and $\eta \in (0, \frac{1}{2})$,  (\ref{tang Hessian}) and \eqref{bnd on HWTL} give us
\begin{equation*}
     0 \leq \hess_\omega(Y,Y)|_{B(p,\eta)}.
    \end{equation*}
  %%%%  \[\hess_\psi(Y,Y)\leq \left( C + \epsilon O\left( \dist_p (x) \right) \right) g(Y,Y). \]
Since $\| \phi' \| \leq \epsilon\eta^3$ and $\supp(\omega) \subset B(p,2\eta)$, (\ref{tang Hessian}) and \eqref{bnd on HWTL} also yield
 \begin{eqnarray*}
 \hess_\omega(Y,Y)|_{M\backslash B(p,\eta)} \geq - 2 \frac{\epsilon \eta^3}{\eta} \geq -\epsilon.
%%%-\epsilon \leq - \frac{\epsilon \eta^2}{2}  = - \frac{\epsilon \eta^3}{2\eta}   \leq  \hess_\omega(Y,Y)|_{M\backslash %%%B(p,\eta)} 
    \end{eqnarray*}
    Hence for all unit vectors  $z\in TM$,
    \[- \epsilon \leq \hess_\omega(Z,Z) \leq 3C,\]
    as claimed in (\ref{Hessian inequal +}).
%%%%%%%%%%%%%%%%%%%%%%%%%%%%%%%%%%%
    
To show that (\ref{Hessian inequal +b}) holds, for 
$Z \in TM|_{B(p,\tau)} $ we note that since for $t\in [0, \tau]$, $\phi''(t)= 2C$,
\begin{equation*}
  \hess_\omega(X,X)|_{B(p,\tau)}  \equiv 2 C,
\end{equation*}
   and $\phi'(t)= 2Ct$ , so (\ref{tang Hessian}) and \eqref{bnd on HWTL} give us that for any unit vector $Y \perp X$ with $Y \in TM|_{B(p,\tau)}$,
%%%%%%%%%%%%%%%%%%%%%%%%%%%%%
     \begin{eqnarray*}
    \hess_\omega(Y,Y) &= &
    \phi'\left(\dist_p (x)\right) \left( \frac{1}{\dist_p (x)} + O\left( \dist_p (x) \right) \right) \\
    &\geq& 2C \dist_p (x) \left( \frac{1}{\dist_p (x)} - \sqrt{\eta} \right)\\
    &> & C,
    \end{eqnarray*}
    since we are at a point where $\dist_p (x) < \tau <\eta < \frac{1}{2}$ and $C >1$.
\end{proof}

Our applications of Lemma \ref{existence lemma omega} start with a  Riemannian  submersion
$$
\pi :M \to B
$$
from a compact manifold  $M$. We then apply Lemma \ref{existence lemma omega} at \red{a} point $p\in B$ to get a function 
$\omega :B \to \R.$ Next we pull the 
pull $\omega$ back via $\pi$ to get a a  $\pi$--basic function
\begin{equation*}
 \tilde{\omega} :M\to \R, \quad \quad
\tilde{\omega}:= \omega \circ \pi.
\end{equation*}
To control the Hessians of $\tilde{\omega}$ we note that the compactness of $M$ gives us that there is a uniform upper bound  
$$
\norm{\sigma} \leq K
$$
for the norms of the second fundamental forms of the fibers of  $\pi$ and that 
$$
d\pi (\grad \tilde{\omega}) |_{x} =  \grad \omega|_{\pi(x)}   =  \phi'\left(\dist_p (\pi (x))\right)X,
$$
where $\phi: \R \to \R$ and $X: =  \grad \dist_p $ are as in the proof of Lemma \ref{existence lemma omega}. 
Exploiting the previous two displays, the Hessian formulas \eqref{Hessian on XX} and \eqref{Hess Y}, and replacing the quantity $\epsilon \eta^3$ in table \eqref{da table} with $\frac{\epsilon \eta^3}{K}$ gives us the following corollary of Lemma \ref{existence lemma omega}. 

\begin{cor}\label{existence lemma omega vertical}
Let $\pi :M\to B$ be a Riemannian submersion from a compact manifold $M$. Given any $p\in B$, $C<-1$, and $\epsilon ,\eta \in (0,\ \frac{1}{2}\min \{1,\ \inj_{M}\})$ and any $\tau \in (0,\frac{\epsilon \eta }{2|C|})$ there is a $\pi$-basic smooth function 
\begin{equation*}
\tilde{\omega} :M \to \R
\end{equation*}
so that 
\begin{enumerate}
\item $\grad \tilde{\omega}  |_{\pi ^{-1}(p)}=0$,

\item $\norm{\tilde{\omega}}_{C^{1}}<\epsilon$, 

\item $\supp(\tilde{\omega})\subset B(\pi ^{-1}(p),2\eta )$,
\end{enumerate}
{\color{red} Moreover, for all unit vectors $Z\in TM$,
\begin{equation}\label{gen Hess omega v}
3C\leq \hess_{\tilde{\omega}  }(Z,Z)\leq \epsilon,
\end{equation}
and for all horizontal unit vectors $Z\in TM|_{B\left( \pi ^{-1}(p),\tau \right)}$,
\begin{equation}\label{spec Hess omega v} 
\hess_{\tilde{\omega}}(Z,Z)\leq C.
\end{equation}}
\end{cor}

%%%%%%%%%%%%%%%%%%%%%%%%%%%%%%%%%%%%%%%%%%%%%%%%%%%
%%%%%%%%%%%%%%%%%%%%%%%%%%%%%%%%%%%%%%%%%%%%%%%%%%%
%%%%%%%%%%%%%%%%%%%%%%%%%%%%%%%%%%%%%%%%%%%%%%%%%%%

\section{Definition of the Metric}\label{4_definition_of_the_metric}
For some $k \geq b$, let 
\begin{equation}\label{g_m}
    \pi:(M^n,g_M)\to(B^b,g_B)
\end{equation}
be a Riemannian submersion with $\ric_k^M \geq k$. 
For $p\in B$, $C_h > 1$, $C_v < -1$, and
\[\epsilon_h, \eta_h, \epsilon_v, \eta_v \in \left(0, \ \frac{1}{2} \min \{ 1, \ \inj_M  \}\right),\]
let 
\[\omega_h : B \to \R \; \; \text{and} \;\; \omega_v : B \to\R \]
be as in Lemma \ref{existence lemma omega} with $\epsilon_h, \eta_h$, and $C_h$ playing the role of $\epsilon, \eta$, and $C$ in the construction of $\omega_h$, and $\epsilon_v, \eta_v$, and $C_v$ playing the role of $\epsilon, \eta$, and $C$ in the construction of $\omega_v$.

Now define
\[\tilde{\omega}_h :M  \to\R \quad\text{and}\quad \tilde{\omega}_v : M \to\R\]
by
\begin{equation}\label{omega h omega v}
    \tilde{\omega}_h  :=  \omega_h \circ\pi \quad\text{and}\quad \tilde{\omega}_v := \omega_v  \circ\pi,
\end{equation}
and set
\begin{equation}\label{warped B metric defn}
   \tilde{g}_B := e^{2\omega_h}g_B
\end{equation}
and 
\begin{equation}\label{warped M metric defn}
   \tilde{g}_M := e^{2\tilde{\omega}_h}g_M|_{\cal H\oplus \cal H} + e^{2\tilde{\omega}_v}g_M|_{\cal V \oplus \cal V}.  
\end{equation}
It immediately follows that
\[
\pi:(M^n,\tilde{g}_M)\to(B^b,\tilde{g}_B)
\]
continues to be a Riemannian submersion. Our goal is to show that there is a choice of $\omega_h$ and $\omega_v$ so that 
\[\tric^{M}_{k} := \ric_k(M, \tilde{g}_M ) >0\]
and 
\[\tric^{B} := \ric(B, \tilde{g}_B )\]
obtains both positive and negative values. 

Theorem 2 of \cite{pro_wilhelm} says that if  $\pi : M \to B$ is a Riemannian submersion from a compact manifold $M$ with $\ric^{M}> 0$, then $B$ cannot have nonpositive Ricci curvature. Thus to prove that $\tric^{B}$ has both signs, we only need show that $\tric^{M}_{k}  >0 $ and that $\tric^{B}$ obtains a negative value. 
%%%%%%%%%%%%%%%%%%%%%%%%%%%%%%%%%%%%%%%%%%%%%%%%%%%
%%%%%%%%%%%%%%%%%%%%%%%%%%%%%%%%%%%%%%%%%%%%%%%%%%%
%%%%%%%%%%%%%%%%%%%%%%%%%%%%%%%%%%%%%%%%%%%%%%%%%%%

\section{Verification on the Base}\label{5_verification_on_the_base}
In this section, we estimate the impact on the curvature of the base of the conformal change
\[
 \tilde{g}_B = e^{2\omega_h}g_B,
\]
defined in (\ref{warped B metric defn}). We write $R^b$ and  $\ric^B$  for the curvature and Ricci tensors of $(B,g_B)$,  and we write $ \tric^{B}$ and $\tilde{R}_B$  for the curvature and Ricci tensors of $(B,\tilde{g}_B)$. We set 
\[
\mathfrak{D} R_B :=  \tilde{R}_B - R_B,
\]
and  we  choose $\omega_h$ so that $ \tilde{g}_B$ satisfies the conclusion of the following result.
\begin{prop}\label{neg Ricci at p}
Given any $\epsilon > 0$, $K>1$, and $\eta_h \in \left(0, \frac{1}{2} \min \{ 1,  \ \inj_M  \} \right)$, there is an $\epsilon_{h}^{0} >0$ so that if $\epsilon_{h}  \in (0,\epsilon_{h}^{0} ) $ and
\begin{equation}\label{choice of C_h}
    C_h > \tilde{K} := \frac{K}{2} + \max \left\{   \abs{\sec^B|_p} \right\} +1,
\end{equation}
then
\begin{equation}\label{Ricci at p}
  \tsec^B|_p < -K  
\end{equation}
and
\begin{equation}\label{Delta R at p}
    \norm{  \mathfrak{D} R_B - (-  2 \hess_{\omega_h} \circ g_B) } < \epsilon, 
\end{equation}
where $\circ$ is the Kulkarni--Nomizu product defined in \eqref{Kulkarni Nomizu}.
\end{prop}
Notice, in particular, that \eqref{Ricci at p} implies that 
\begin{equation}\label{Ricci at p--2}
    \tric^{B} (\cdot, \cdot)|_p < -K (b-1) g(\cdot, \cdot)|_p.
\end{equation}
\begin{remark-nonumb}
For  a similar result on conformal changes supported on a tubular neighborhood of a compact submanifold, see 
Theorem 2.1 in \cite{searle_wilhelm}.
\end{remark-nonumb}
\begin{proof}
By  Parts 2 and 3 of Exercise 4.7.14 in  \cite{petersen_book}, 
\begin{eqnarray}\label{curv op conf}
e^{-2\omega_h}  \tilde{R}_B &=& R_B - \left(2 \hess^{B}_{\omega_h} -2 (\odif{\omega_h}\otimes\odif{\omega_h}) + |d \omega_h|^2 g_B\right) \circ g_B, 
\end{eqnarray}
and for $g_B$-orthonormal $X,Y\in TB$, 
\begin{eqnarray}\label{sec conf}
e^{2\omega_h}  \tsec_B(X,Y)   &=&   \sec_B(X,Y)   - \hess^{B}_{\omega_h} (X,X) - \hess^{B}_{\omega_h} (Y,Y) \nonumber \\
                             &+& (D_X \omega_h)^2 + (D_Y \omega_h)^2 - |d \omega_h|^2.
\end{eqnarray}
However, by Part 1 of Lemma \ref{existence lemma omega}, $\nabla \omega_h |_p =0$. Thus at $p$, 
\begin{equation}\label{curv at p}
e^{2\omega_h}  \tsec_B(X,Y)|_p   =   \sec_B(X,Y)|_p   - \hess^{B}_{\omega_h} (X,X)|_p - \hess^{B}_{\omega_h} (Y,Y)|_p.
\end{equation}
Noting that $C_h > \tilde{K}$, Inequality (\ref{Hessian inequal +b}) gives us that for all unit vectors $Z\in T_pM$,
\[
\hess^{B}_{\omega_h}(Z, Z)|_p > \tilde{K}, 
\]
which combined with \eqref{choice of C_h}  and \eqref{curv at p} gives us \eqref{Ricci at p}. 

To prove \eqref{Delta R at p}, we note that since $\norm{ \omega_h }_{C^1} < \epsilon_h$, \eqref{curv op conf} simplifies to
\begin{equation}
\norm{  \frak{D} R_B - (-  2  \hess_{\omega_h} \circ g_B) } <   O(\epsilon_h ) < \epsilon,
\end{equation}
provided $\epsilon_{h}^{0}$ is sufficiently small. 
\end{proof}

%%%%%%%%%%%%%%%%%%%%%%%%%%%%%%%%%%%%%%%%%%%%%%%%%%%
%%%%%%%%%%%%%%%%%%%%%%%%%%%%%%%%%%%%%%%%%%%%%%%%%%%
%%%%%%%%%%%%%%%%%%%%%%%%%%%%%%%%%%%%%%%%%%%%%%%%%%%

\section{Method to Verify Positive Intermediate Ricci Curvature}\label{6_methods_to_verify_PIRC}
In this section we prove a lemma  to estimate changes in intermediate Ricci 
curvature caused by changes in the curvature operator. In large part, our result relies on a result of Reiser and Wraith from \cite{reiser_wraith}.
Before stating either result we introduce some notation.

Let $g$ and $\tilde{g}$ be two Riemannian metrics on a smooth $n$-manifold $M$ with  sectional curvatures $\sec^M $  and $\tsec^M$.
To understand the change in the intermediate Ricci curvature lower bound when we transform $(M, g_M)$ to $(M, \tilde{g}_M)$,  we will say that 
    \[
    \frak{D} \ric^M_k \geq r,
    \]
provided for all $g_M$-orthonormal sets $ \{u, v_1,\cdots,v_k\} $  we have
    \[
    \sum_{i=1}^k \tsec^M (u, v_i) -\sec^M (u, v_i)  > r. 
    \]

In Section \ref{7_analysis_on_the_total_space}, we will complete the proof of Theorem \ref{trying too hard} by employing the following lemma to verify that there is a choice of $\omega_h$ and $\omega_v$ so that the metric
\begin{equation*}%%\label{warped M}
   \tilde{g}_M := e^{2\tilde{\omega}_h}g_M|_{\cal H\oplus \cal H} + e^{2\tilde{\omega}_v}g_M|_{\cal V \oplus \cal V}
\end{equation*}
has positive $\ric_k(M, \tilde{g}_M)$, while the induced metric 
\begin{equation*}%%%\label{warped B metric }
   \tilde{g}_B = e^{2\omega_h}g_B,  
\end{equation*}
on  $B$ has Ricci curvature of both signs. 
\begin{lemma}
\label{verification cor}
Given $\epsilon>0$ and $b\in\N$, there is a  $\delta>0$ with the following property.
Let $g$ and $\tilde{g}$ be two complete Riemannian metrics on a smooth $n$-manifold $M$ with curvature operators $\cal R_M$ and $\widetilde{\cal R}_M$, respectively. Let $\cal H$ and $\cal V$ be $g$-orthogonal distributions so that
\[ 
TM = \cal H  \oplus \cal V 
\]
with $\dim \cal H = b$. Set 
\[
\frak{D}\cal{R}_M := \widetilde{\cal R}_M - \cal R_M.
\]
Suppose that for  all $g$-orthonormal pairs  $\{X, Y\}\subset \cal H$ and $\{V, W\} \subset \cal V$, 
\begin{eqnarray}
%%%%%%%%%%%%%%%%%%%%%%%%%%%%%%%%%%%%%%5
\label{abstract condition for closeness}
     \norm{\frak{D}\cal R \cal(V\wedge W) } &<&\delta \label{abstract condition for closeness--2}\\
     \abs{ \frak{D}\cal{R}_M\left( X\wedge Y\right)^{( \cal H \wedge \cal V ) \oplus ( \cal V \wedge \cal V)}}  &<&\delta,  \label{XY alm eigen} \;\; \text{and} \\
\abs{  \frak{D}\cal{R}_M \left( X\wedge V\right)^{( \cal H \wedge \cal H ) \oplus ( \cal V \wedge \cal V)}}  &<&\delta, 
\end{eqnarray}
%%%%%%%%%%%%%%%%%%%%%%%%%%%%%%%%%%%%%%%%%%
where the superscripts $^{( \cal H \wedge \cal V ) \oplus ( \cal V \wedge \cal V)}$ and $^{( \cal H \wedge \cal H ) \oplus ( \cal V \wedge \cal V)}$ denote the components of the vectors in $( \cal H \wedge \cal V ) \oplus ( \cal V \wedge \cal V)$ and $( \cal H \wedge \cal H ) \oplus ( \cal V \wedge \cal V)$, respectively. 

Suppose further that for some $\lambda_{h,h},\lambda_{h,v}\in\R$,
\begin{eqnarray}
 g(\frak{D}\cal{R}_M\left( X\wedge Y\right) ,X\wedge Y) &\geq &\lambda
_{h,h}, \\
g(\frak{D}\cal{R}_M\left( X\wedge V\right) ,X\wedge V) &\geq &\lambda
_{h,v},  \label{just better than--XV}    
\end{eqnarray}
and either 
\begin{equation}\label{RW Sum condition}
    \lambda_{h,v}   >0\quad\text{and}\quad  (b-1)\lambda_{h,h} + \lambda_{h,v}> 0
    \end{equation}
    or 
\begin{equation}\label{all small}
  \abs{\lambda_{h,h}}  <\delta \quad\text{and}\quad \abs{\lambda_{h,v}} <\delta.
\end{equation} 
Then  for any  $k\in \{b,b+1,\ldots, n-1\}$,
\[
\frak{D} \ric^M_k >  - \epsilon.
\]
\end{lemma}

The proof involves several preliminary results, one of which is the  following linear algebra lemma of Reiser and Wraith from \cite{reiser_wraith}.

\begin{lemma}[Reiser--Wraith, \cite{reiser_wraith}]\label{reiser_wraith} Let $V$ be a finite dimensional real vector space with a fixed inner product  $\langle \cdot, \cdot \rangle$.
    Let $A: \Lambda^2 V \rightarrow \Lambda^2 V$ be a self-adjoint linear map and let $k\in\{1,\cdots,\dim(V)-1\}$. Suppose that $V$ splits orthogonally as
    $V=V_1 \oplus V_2 \oplus V_3$, so that the spaces
    \[ V_i \wedge V_j \]
    with $i, j \in\{1,2,3\}$ are eigenspaces of $A$ with eigenvalues $\lambda_{i j}$. If for every $i \in\{1,2,3\}$ and every $n_{i1},n_{i2},n_{i3} \in \{ 0,1,2,3, \ldots, \dim(V) \}$ with $n_{ij} \leq \operatorname{dim}\left(V_j\right)(j \neq i)$, $n_{ii} \leq \operatorname{dim}\left(V_i\right)-1$, and 
    \begin{equation}\label{RW Sum}
        n_{i1}+n_{i2}+n_{i3}=k
    \end{equation}
    we have
    \begin{equation}\label{reiser_wraith_eigenvalues}
       n_{i1}\lambda_{i 1} + n_{i2}\lambda_{i 2} + n_{i3}\lambda_{i 3} > c, 
    \end{equation}
    for some $c\in\R$, then for every orthonormal set $ \{u_0, \cdots, u_k\} $ in $ V $,
    \begin{equation*}
       \sum_{i=1}^k \inn{A(u_0 \wedge u_i), u_0 \wedge u_i} > c.
    \end{equation*}
\end{lemma}

The following lemma is an easy consequence of Lemma \ref{reiser_wraith}.
\begin{lemma}\label{reiser_wraith_contextual}
 Let $A: \Lambda^2 V \rightarrow \Lambda^2 V$ be as in Lemma \ref{reiser_wraith}, with $V_3 = \{ 0\}$, and $\dim(V_1) =b$. 
 Given $\epsilon>0$, there is a $\delta>0$ with the following property: Suppose that 
 \begin{equation}\label{reiser_wraith_contextual eqn} 
    \lambda_{12} > 0, \; \; \;    (b-1) \lambda_{11}+ \lambda_{12} > 0 , \; \; \; \text{and} \; \;  | \lambda_{22} | < \delta.
    \end{equation}
Then for every  $k \geq b$ and every orthonormal set $ \{u_0, u_1 \cdots, u_k\} $ in $ V $,
    \begin{equation}\label{reiser_wraith_inner_product}
       \sum_{i=1}^k \inn{A(u_0 \wedge u_i), u_0 \wedge u_i} > -\epsilon.
    \end{equation}
\end{lemma}
\begin{proof}
    Let $n_{ij}$ be as in Lemma \ref{reiser_wraith}. Since $V_{3}=\{0\},$ $n_{13}=n_{23}=n_{33}=0$ and we only need consider the cases when $i=\{1,2\}.$
When $i=1,$ (\ref{RW Sum}) becomes 
\begin{equation*}
n_{11}+n_{12} = k.
\end{equation*}
Since $k\geq b$ and $n_{11}\leq b-1$, it follows that $n_{12}\geq 1.$

If $\lambda _{11}>0,$ then since $\lambda _{12}>0$, 
%%%%%%%%%%%%%%%%%%%%%%%%%
\begin{equation}\label{trivial RW}
n_{11}\lambda _{11}+n_{12}\lambda _{12}>0.
\end{equation}
%%%%%%%%%%%%%%%%%%%%%%%%%%%%%%%%%
Hence, we may assume that $\lambda _{11}\leq 0$.
This together with $\lambda _{12}>0$, $n_{12}\geq 1$,  and $n_{11}\leq b-1$ gives us    
\begin{equation}\label{easy RW}
n_{11}\lambda _{11}+n_{12}\lambda _{12}\geq\left( b-1\right) \lambda
_{11}+\lambda _{12}>0,
\end{equation}
by \eqref{reiser_wraith_contextual eqn}.

When $i=2,$ since $\lambda _{21}>0$ we have 
\begin{equation}\label{pretty easy RW}
n_{21}\lambda _{21}+n_{22}\lambda _{22}>n_{22}\lambda _{22}>-|\delta
n_{22}|>-\epsilon ,
\end{equation}
provided we choose $\delta \in \left( 0,\frac{\epsilon }{\dim(V) }\right)$. Together (\ref{trivial RW}), (\ref{easy RW}), and (\ref{pretty easy RW}) give us that (\ref{reiser_wraith_eigenvalues}) holds in all cases with $c=-\epsilon $, hence (\ref{reiser_wraith_inner_product}) holds, as desired.
\end{proof}
%%%%%%%%%%%%%%%%%%%%%%%%%%%%%%%
In addition to Lemma \ref{reiser_wraith_contextual}, the proof of Lemma \ref{verification cor} uses the following observation, which is a special case of the fact that the norm of a self-adjoint map is small provided all of its eigenvalues are small. 
\begin{lemma}\label{reiser_wraith_contextual 2}
    Let $A: \Lambda^2 V \rightarrow \Lambda^2 V$ be as in Lemma \ref{reiser_wraith}, with $V_3 = \{ 0\}$ and $\dim (V_1) =b$. 
 If 
 \begin{equation}\label{reiser_wraith_contextual_2 eqn} 
      | \lambda_{11}| < \delta, \; \;   |\lambda_{12} |< \delta, \; \; \text{and} \; \;  | \lambda_{22} | < \delta,
    \end{equation}
then
\[\norm{A}  \leq \delta.\]
\end{lemma}
\begin{proof}
   A unit vector $\mathfrak{u}\in $ $\wedge ^{2}V$ has the form
\begin{equation*}
\mathfrak{u}=\mu _{11}{\color{red}\mathfrak{e}}_{11}+\mu _{12}{\color{red}\mathfrak{e}}_{12}+\mu _{22}{\color{red}\mathfrak{e}}_{22},
\end{equation*}
where $(\mu _{11})^{2}+(\mu _{12})^{2}+(\mu_{22})^{2}=1$ and ${\color{red}\mathfrak{e}}_{11}\in V_{1}\wedge V_{1},$ ${\color{red}\mathfrak{e}}_{12}\in V_{1}\wedge V_{2},$ ${\color{red}\mathfrak{e}}_{22}\in V_{2}\wedge V_{2}$ are unit vectors. 
Since $V_{1}\wedge V_{1}$, $V_{1}\wedge V_{2}$, and $V_{2}\wedge V_{2}$ are orthogonal, 
\begin{eqnarray*}
\abs{\inn{A(\mathfrak{u}), A(\mathfrak{u})}} &=& 
  \norm { \mu_{11}\lambda _{11}{\color{red}\mathfrak{e}}_{11} + \mu _{12}\lambda _{12}{\color{red}\mathfrak{e}}_{12} + \mu_{22}\lambda_{22}{\color{red}\mathfrak{e}}_{22} }^2 \\
& =& \lambda _{11}^2 \mu _{11}^{2}+\lambda _{12}^2\mu _{12}^{2}+\lambda _{22}^2\mu
_{22}^{2}  \\
&\leq& \delta^2 \left( \mu _{11}^{2}+\mu _{12}^{2}+\mu _{22}^{2}\right)  \\
&=&\delta^2,
\end{eqnarray*}
as claimed.
\end{proof}

\begin{proof}[Proof of Lemma \ref{verification cor}]
Let $\lambda _{h,h}, \lambda _{h,v} $ be as in Lemma \ref{verification cor}, and define
\begin{eqnarray*}
 {\cal A} :  TM \wedge TM &\longrightarrow&   TM \wedge TM \; \; \; \text{by setting}\\
%%%%%%%%%%%%%%%%%%%%%%%%%%%%%%%%%%%%%%%%%%%%%%
\mathcal{A}(X\wedge Y) &=&\lambda _{h,h}\left( X\wedge Y\right), \; \; \text{for all $X,Y \in \cal{H}$} \\
\mathcal{A}(X\wedge V) &=&\lambda _{h,v}\left( X\wedge V\right), \; \; \text{for all $X\in \cal{H}, Y\in \cal{V}$, and } \\
\mathcal{A}(V\wedge W) &=&0 ,    \; \; \text{for all $V,W \in \cal{V}$.}
\end{eqnarray*}

Informally, Inequalities (\ref{abstract condition for closeness})--\eqref{just better than--XV} mean that  $\frak{D}\cal{R}_M$ is almost $\geq\mathcal{A}$. More precisely, 
given any $\epsilon>0$, there is a $\delta>0$ so that if  (\ref{abstract condition for closeness})--\eqref{just better than--XV} hold for our choice of $\delta>0$, then 
%%%%%%%%%%%%%%%%%%%%%%%%
for all $p\in M$ and all $g$-orthonormal $(k+1)$-frames $\{ u, v_1, \ldots, v_k \} \subset T_pM$, 
\begin{equation}\label{model is smaller}
\sum_{i=1}^k g(\frak{D}\cal{R}_M (u\wedge v_i), u\wedge v_i ) > \sum_{i=1}^kg( \cal{A} (u\wedge v_i), u\wedge v_i) - \frac{\epsilon}{2}.
\end{equation}

Thus if (\ref{RW Sum condition}) holds, then by combining (\ref{model is smaller}) with Lemma \ref{reiser_wraith_contextual} and possibly choosing an even smaller $\delta>0$ we get that 
\begin{eqnarray}\label{da end}\nonumber
   \sum_{i=1}^k g(\frak{D}\cal{R}_M (u\wedge v_i), u\wedge v_i ) 
   &\geq& \sum_{i=1}^kg( \cal{A} (u\wedge v_i), u\wedge v_i) - \frac{\epsilon}{2} \\
   &\geq& - \frac{\epsilon}{2} - \frac{\epsilon}{2},
\end{eqnarray}
as desired. 
Similarly, if (\ref{all small}) holds, then we combine (\ref{model is smaller}) with Lemma \ref{reiser_wraith_contextual 2} and also get that (\ref{da end}) holds, after possibly making $\delta$ even smaller.
\end{proof}

%%%%%%%%%%%%%%%%%%%%%%%%%%%%%%%%%%%%%%%%%%%%%%%%%%%
%%%%%%%%%%%%%%%%%%%%%%%%%%%%%%%%%%%%%%%%%%%%%%%%%%%
%%%%%%%%%%%%%%%%%%%%%%%%%%%%%%%%%%%%%%%%%%%%%%%%%%%

\section{Analysis on the Total Space}\label{7_analysis_on_the_total_space}
In this section we complete the proof of Theorem \ref{main-theorem} by showing that there is a choice of functions $\tilde{\omega}_h$ and $\tilde{\omega}_v$, defined using Lemma \ref{existence lemma omega}, so that the resulting metric 
\begin{equation*}%%\label{pet and repeat--b}
   \tilde{g}_M := e^{2\tilde{\omega}_h}g_M|_{\cal H\oplus \cal H} + e^{2\tilde{\omega}_v}g_M|_{\cal V \oplus \cal V}
\end{equation*}
has positive intermediate Ricci curvature, while $\tric^{B}$ obtains both signs. In fact, we will verify Inequality \eqref{almost the same} and hence also complete the proof of Theorem \ref{trying too hard}. 

We will compare the curvatures of $\tilde{g}_M$ with those of  $g_M$ by comparing both to those of the following auxiliary metric:
%%%%%%%%%%%%%%%%%%%%%%%%%%%%%%%%%%%
\begin{eqnarray}\label{hat g}
\hat{g}_M &= & \;\; \pi^*(\tilde{g}_B)|_{\cal H \oplus \cal H} \;\, +  \;g_M|_{\cal V \oplus \cal V} \nonumber \\
          & =&  e^{2\tilde{\omega}_h} g_M |_{\cal H \oplus \cal H} \;\;\; +  \;  g_M|_{\cal V \oplus \cal V}.
\end{eqnarray}
%%%%%%%%%%%%%%%%%%%%%%%%%%%%%%%%%%%
Note that \red{the}  submersion $\pi :M\to B$ is Riemannian with respect to both metrics $\hat{g}_M$ and $\tilde{g}_M$, and that in either case, the induced metric on $B$ is $\tilde{g}_B$. To ease the exposition, we set 
$$
\hat{g}_B := \tilde{g}_B.
$$
%%%%%%%%%%%%%%%%%%%%%%%%%%%%%

The comparison between the curvatures of $\hat{g}_M $ and $\tilde{g}_M$ will be based on the Gromoll--Walschap formulas for the changes in the curvature tensor of a Riemannian submersion when the metric on the fibers is changed by a warping function of the base (cf  (\ref{warped M metric defn}) and (\ref{hat g})).
%%%%%%%%%%%%%%%%%%%%%%%%%%%%%%%%%%%%%%%%%%%%%%

We denote the covariant derivative of $g_M$ by $\nabla$. We add a tilde or a hat above a covariant derivative  or a curvature to denote its association to the metrics $\tilde{g}_M$ and $\hat{g}_M$, respectively.

We start by recording the following general proposition on the impact of a $C^1$-small change to a fixed metric  on the metric's covariant derivative  and the gradient and Hessian of a fixed real valued function. 

%%%%%%%%%%%%%%%%%%%%%%%%%%%%%%%%%%%%%%%%%%%%%%
\begin{prop}\label{things-are-small-lemma-1}
	Let $(M,g)$ be a compact Riemannian manifold $M$. Let $\red{\cal A}$ be a fixed finite atlas for $M$ that satisfies \eqref{extended chart}, and let $f:M\to\R$ be a fixed $C^\infty$ function on $M$. 
 
 For every $\epsilon >0$, there is a $\delta >0$, that also depends on $g$, $\red{\cal A}$ and $f$, with the following property. If $\hat{g}$ is a Riemannian metric  on $M$ that satisfies 
	\begin{equation}\label{change-in-g-is-small}
	   \norm{g-\hat{g}}_{C^{1}, g_0}<\delta, 
	\end{equation}
then 
%%%%%%%%%%%%%%%%%%%%%%%%%%%%%%%%%%%%%%%%%%%%%%
\begin{equation}\label{change-in-grad-is-small}
	    \norm{\nabla f-\hat\nabla f}_g<\epsilon,
	\end{equation}
 %%%%%%%%%%%%%%%%%%%%%%%%%%%%%%%%%%%%%%%%%%%%%%
\begin{equation}\label{change-in-hessian-is-small}
	    \norm{\hess_f-\widehat\hess_f}_g<\epsilon, \hspace{0.08in} \text{and }
	\end{equation}
%%%%%%%%%%%%%%%%%%%%%%%%%%%%%%%%%%%%%%%%%%%%%%
\begin{equation}\label{change-in-covariant-is-small}
	    \norm{\nabla_{\boldsymbol{\cdot}} E-\hat\nabla_{\boldsymbol{\cdot}} E}_g <\epsilon 
	\end{equation}
 for any coordinate field $E$ of any chart of $\red{\cal A}$.  
\end{prop}
%%%%%%%%%%%%%%%%%%%%%%%%%%%%%%%%%%%%%%%%%%%%%%
\begin{proof}
	For all $y \in TM$, 
 %%%%%%%%%%%%%%%%%%%%%%%%%%%%%%%%%%%%%%%%%%%%%%
	\begin{eqnarray}\label{dfn of grad}
	   g( \nabla f,y) =df(y) =\hat{g}(\hat\nabla f,y).  
	\end{eqnarray}
  %%%%%%%%%%%%%%%%%%%%%%%%%%%%%%%%%%%%%%%%%%%%%%
  %%%%%%%%%%%%%%%
This gives us 
	\begin{align}
	\abs{g( \nabla f,y) - g( \hat\nabla f,y)} 
				\;& \leq \; \abs{g(\nabla f,y) - \hat g(\hat\nabla f,y)}
					\;+\; 	\abs{ \hat g(\hat\nabla f,y) -  g(\hat\nabla f,y)}  \nonumber \\
				\;& \leq \; 0 \;+\; \abs{(g-\hat g)\left(\hat\nabla f,y\right)}, \hspace{0.09in} \text{by} \; (\ref{dfn of grad} ). 
			 \label{grad diff}
	\end{align}
	
Since $M$ is compact, there is a $K>0$ so that 
\begin{equation*}
\max \norm{\hat{\nabla}f}_{\hat{g}}\leq K.
\end{equation*}
This together with the fact that $\norm{g-\hat{g}}_{C^{0},g_0}<\delta$
implies that 
\begin{equation*}
\max \norm{\hat{\nabla}f}_{g}\leq K\left( 1+\delta \right), 
\end{equation*}
and for $\norm{y}_{g}=1$, we have $\norm{y}_{\hat{g}}\leq 1+\delta$. Thus for $\norm{y}_{g}=1,$
\begin{equation*}
\abs{ \left( g-\hat{g}\right) \left( \hat{\nabla}f,y\right) } \leq K(1+\delta)^2 \delta < 4K\delta,
\end{equation*}
provided $\delta <1$. So
\begin{eqnarray*}
   \abs{g( \nabla f,y) - g( \hat\nabla f,y)}&<&4K\delta\\
                                            &<& \epsilon,
\end{eqnarray*}
provided we choose $\delta < \frac{\epsilon}{4K}$, which completes the proof of  
(\ref{change-in-grad-is-small}).

 To prove (\ref{change-in-covariant-is-small}) we write 
 \begin{equation}\label{tri ineq}
 \begin{split}
     \abs{g( \nabla_X E, Y) - g( \hat \nabla_X E, Y)}
				\leq \; & \abs{g(\nabla_X E, Y) -  \hat g( \hat\nabla_X E, Y)}\\
                &+ \abs{( \hat g-  g)(\hat\nabla_X E, Y)}.
 \end{split}
 \end{equation}
Now suppose that  $X$ and $Y$ are coordinate fields for $\red{\cal A}$. Since  $\red{\cal A}$ satisfies \eqref{extended chart} and there are only a finite number of possible $X$ and $Y$, the \red{Koszul} formula together with Proposition \ref{two norms prop} and  (\ref{change-in-g-is-small}) gives us that 
 %%%%%%%%%%%%%%%%%%%%%%%%%%%%%%%%%%%%%%%%%%%%%%%%%
\begin{equation}\label{1st term}
\abs{  g(\nabla_X E, Y) - \hat g(\hat\nabla_X E, Y)} <  \frac{\epsilon}{3}
\end{equation}
and
\begin{equation}\label{2nd term}
\abs{( \hat g- g)(\hat \nabla_X E, Y)} < \frac{\epsilon}{3},
\end{equation}
provided $\delta$ is sufficiently small. Combining \eqref{tri ineq}, \eqref{1st term}, and \eqref{2nd term} with \eqref{compactness} gives us (\ref{change-in-covariant-is-small}). 
 %%%%%%%%%%%%%%%%%%%%%%%%%%%%%%%%%%%%%%%%%%%%%%%%%
	
	To prove \eqref{change-in-hessian-is-small} we note that
	\begin{align*}
	\hess_f(X,Y) & = g(\nabla_X \nabla f, Y) \\
				 & = X g(\nabla f, Y) - g(\nabla f, \nabla_X Y) \\
				 & = X df(Y) - g(\nabla f, \nabla_X Y),
	\end{align*}
	and similarly, 
  %%%%%%%%%%%%%%%%%%%%%%%%%%%%%%%%%%%%%%%%%%%%%%%%%
	\[ 
 \widehat\hess_f(X,Y) = X df(Y) - \hat g(\hat\nabla f, \hat\nabla_X Y).
 \]
  %%%%%%%%%%%%%%%%%%%%%%%%%%%%%%%%%%%%%%%%%%%%%%%%%
	In particular, 
  %%%%%%%%%%%%%%%%%%%%%%%%%%%%%%%%%%%%%%%%%%%%%%%%%
\[ 
 \hess_f(X,Y) - \widehat\hess_f(X,Y) = 
    \hat g(\hat\nabla f, \hat\nabla_X Y) - g(\nabla f, \nabla_X Y).
\]
    %%%%%%%%%%%%%%%%%%%%%%%%%%%%%%%%%%%%%%%%%%%%%%%%%
	Thus for coordinate fields $X$ and $Y$, 
      %%%%%%%%%%%%%%%%%%%%%%%%%%%%%%%%%%%%%%%%%%%%%%%%%
	\begin{align*}
 \abs{ (\hess_f - \widehat\hess_f) (X,Y)}  
                 \leq\; & \abs{\hat g(\hat\nabla f, \hat\nabla_X Y) - g(\hat\nabla f, \hat\nabla_X Y)}\\
                & + \abs{ g(\hat\nabla f, \hat\nabla_X Y) -  g(\nabla f, \hat\nabla_X Y)}  \\
			& + \abs{g(\nabla f, \hat\nabla_X Y) - g(\nabla f, \nabla_X Y)}   \\
                 = \; & \abs{(\hat g-g)(\hat\nabla f, \hat\nabla_X Y)} \\
                & + \abs{g(\hat\nabla f-\nabla f, \hat\nabla_X Y)} \\
                & + \abs{g(\nabla f, \hat\nabla_X Y - \nabla_X Y)}.
	\end{align*}
     %%%%%%%%%%%%%%%%%%%%%%%%%%%%%%%%%%%%%%%%%%%%%%%%%
Combining (\ref{change-in-g-is-small}), (\ref{change-in-grad-is-small}), and (\ref{change-in-covariant-is-small}) with \eqref{compactness} gives us
\[\abs{ (\hess_f - \widehat\hess_f) (X,Y)} < \epsilon,\]
provided $\delta$ is small enough.
\end{proof}

Since the Gromoll--Walschap formulas are in terms of the fundamental tensors $A$, $S$, and $\sigma$ of the submersion $M \longrightarrow B$, we also need the following proposition which records how these tensors change when we replace $g_M$ with $\hat{g}_M$.
\begin{prop}\label{things-are-small-lemma-2}
	Let $g_M$ and $\hat g_M$ be as above. Then for any $X,Y\in\cal H$ and $U,V\in\cal V$,

    \begin{align}
        A_X Y &= \hat{A}_X Y,\label{no-change-in-A}\\
        S_X U &= \hat{S}_X U,\label{no-change-in-S}\\
        A^*_X U &= e^{2\tilde{\omega}_{h}}\hat{A}^*_X U,\label{change-in-A-star-is-small}
        \intertext{and}
        \sigma(U,V) &= e^{2 \tilde{\omega}_{h}}\hat\sigma(U,V),\label{change-in-sigma-is-small}
    \end{align}
    where the tensors $A$, $S$, and $\sigma$ are defined in (\ref{A-tensor}), (\ref{S-tensor}), and (\ref{sigma-tensor}), respectively. 
    In particular, since $\tilde{\omega}_{h}$ is $C^1$-small, the tensors $\hat\sigma$ and $\hat A^*$ are $C^1$-close to their $g_M$-counterparts.
\end{prop}
\begin{proof}
    Equation (\ref{no-change-in-A}) follows from the fact that
    %%%%%%%%%%%%%%%%%%%%%%%%%%%%%%%%%%%
    \[
    A_X Y = \frac{1}{2}[X,Y]^\v = \hat A_X Y.
    \]
    %%%%%%%%%%%%%%%%%%%%%%%%%%%%%%%%%%%%%%
  Dualizing this we find that
    {\color{red}
    \begin{align*}
        g_{M}( A_X^* U,Y) 
        & = g_{M}( U,A_{X}Y) 
        = g_{M}( U,\hat{A}_{X}Y) \\
        & = \hat{g}_{M}( U,\hat{A}_{X}Y) 
        = \hat{g}_{M}( \hat{A}_{X}^* U,Y) \\
        & = e^{-2\tilde{\omega}_{h} }g_{M}( \hat{A}_{X}^*U,Y). 
    \end{align*}}
    As this holds for all horizontal vectors $Y$,
    \begin{equation*}
    A_{X}^{\ast }T=e^{-2\tilde{\omega}_{h} }\hat{A}_{X}^{\ast }U, \;\; \text{proving (\ref{change-in-A-star-is-small}).}
    \end{equation*}
    For vertical vectors $V \in {\cal V}$, $g_M (V, \cdot ) = \hat{g}_M (V, \cdot ) $, so by  
    the \red{Koszul} formula, 
    \begin{equation} \label{same S}
    g_M(S_X V, U) = \hat g_M(\hat S_X V, U), \; \; \text{for all $U,V\in {\cal V} $}.  
    \end{equation}
Since $g_M ( \cdot, U ) = \hat{g}_M (\cdot ,U) $ for all $U\in  {\cal V}$, this proves Equation (\ref{no-change-in-S}).

Dualizing we get
{\color{red}
\begin{align*}
        g_{M}( \sigma(U,V), X)  
        &= g_M(S_X V, U) 
        = \hat{g}_M(\hat S_X V, U)    \;\;\; \ \text{(by (\ref{same S}))} \\
        &=  \hat{g}_{M}( \hat{\sigma}(V,U), X) 
        =  e^{2\tilde{\omega}_{h}}g_{M}( \hat{\sigma}(V,U), X). 
\end{align*}}
    As the previous display is valid for all horizontal vectors $X,$ 
    \begin{equation*}
    \sigma ( U,V ) =e^{2\tilde{\omega}_{h} }\hat{\sigma}(U,V),
    \end{equation*}%
    proving (\ref{change-in-sigma-is-small}). 
\end{proof}

Combining the previous result with the Gray--O'Neill formulas for the curvatures of the total space of a Riemannian submersion immediately gives us the following comparison between the $(1,3)$ curvature tensors $R$ and $\hat{R}$ (see \cite{Gray}, \cite{oneill}, or Theorem 1.5.1 in \cite{gromoll_walschap}).  
\begin{cor}\label{no biggie}
    For every $\epsilon >0$, there is a $\delta >0$ so that if 
   \begin{equation*}
\| \tilde{\omega}_h \|_{C^1} < \delta,
 \end{equation*}  
then  the curvature tensors of $g_{M}$ and $\hat{g}_{M}$ satisfy 
\begin{equation*}
    \norm{ \Big(\hat{R}_M - \pi^*(\hat{R}_B)\Big)\Big|^{\h}_{\cal{H}\oplus\cal{H}\oplus\cal{H}} 
    - \Big(R_M - \pi^*(R_B)\Big)
    \Big|^{\h}_{\cal{H}\oplus\cal{H}\oplus\cal{H}}}
    <\epsilon
\end{equation*}
and
\[
\norm{ \left(\hat{R}_M - R_M\right)\Big|^{\v}_{\cal{H}\oplus\cal{H}\oplus\cal{H}}} < \epsilon,
\]
and when we evaluate $\norm{\hat{R}_M-R_M}$ on any other triple of unit vectors in $%
\mathcal{H}\cup \mathcal{V}$, the result has length $<\epsilon$.
\end{cor}

The following result from Gromoll and Walschap's book (\cite{gromoll_walschap}) describes the change in curvature tensor of the total space of a Riemannian submersion when the metric on the vertical distribution is warped by a function on the base. When the warping function is arbitrary, these equations (\eqref{GW-HHHv}--\eqref{GW-VVVv} below) are rather complicated. However, in our planned application, $\tilde{\omega}_v$ plays the role of the function $\tilde{\phi}: M \to (0,\infty)$ in Proposition \ref{gromoll-walschap-prop}. Since $\tilde{\omega}_v$ is  $C^1$-small, nearly every term on the right side of Equations (\ref{GW-HHHv})--(\ref{GW-VVVv}) will be very small.

%%%%%%%%%%%%%%%%%%%%%%%%%%%%%%%%%%%%%%%%%%%%%%%%%%%%%%%%%%%%%%%%%%%%%%
\begin{prop}\label{gromoll-walschap-prop}
	Let $\pi:(M^n,\hat{g}_M)\to(B^b,\hat{g}_B)$ be any Riemannian submersion, let $\phi: B \to (0,\infty)$ be $C^{\infty}$, and let $\tilde{\phi} := \phi\circ\pi$ be its pullback to $M^{\red{n}}$. Set 
 %%%%%%%%%%%%%%%%%%%%%%%%%%%%%%
	\[ 
 \tilde{g}_M = \pi^*(\hat{g}_B)|_{\cal H \oplus \cal H} 
             + e^{2\tilde{\phi}}\hat{g}_M|_{\cal V\oplus \cal V}.
 \]
	 %%%%%%%%%%%%%%%%%%%%%%%%%%%%%%
  
	Let $X,Y,Z\in\cal H$ and $T,T_1,T_2,T_3\in\cal V$, and let $\inn{\cdot,\cdot}$ denote the inner product associated with the metric $\hat{g}_M$. The components of the curvature tensor associated with $\tilde{g}_M$ are as follows:
	 %%%%%%%%%%%%%%%%%%%%%%%%%%%%%%
     \begin{align}\label{GW-HHHv}	%HHH v
         \begin{split}
             \tilde{R}^\v_M(X,Y)Z =\;& \hat{R}^\v_M(X,Y)Z 
                             + \inn{\hat{\nabla}\tilde{\phi}, X}\hat{A}_Y Z \\
                             &- \inn{\hat{\nabla}\tilde{\phi}, Y}\hat{A}_X Z
                             - \inn{\hat{\nabla}\tilde{\phi}, Z}\hat{A}_X Y,
         \end{split}
    \end{align}
    
	 %%%%%%%%%%%%%%%%%%%%%%%%%%%%%%
	\begin{equation}\label{GW-HHHh}	%HHH h
		\tilde{R}^\h_M(X,Y)Z = e^{2\tilde{\phi}}\hat{R}^\h_M(X,Y)Z 
                             + (1-e^{2\tilde{\phi}})\hat{R}_B(X,Y)Z,
	\end{equation}
	 %%%%%%%%%%%%%%%%%%%%%%%%%%%%%%
     
	\begin{align}\label{GW-HVHv}		%HVH v
	\begin{split}
		\tilde{R}^\v_M(X,T)Y  =\; & \hat{R}^\v_M(X,T)Y 
                    + (1-e^{2\tilde{\phi}}) \hat{A}_X \hat{A}^*_Y T \\
				 & + \left( \widehat{\hess}^M_{\tilde{\phi}}(X,Y) 
                    +\inn{\hat{\nabla}\tilde{\phi},X}\inn{\hat{\nabla}\tilde{\phi},Y} \right) T \\
                  & - \left( \inn{\hat{\nabla}\tilde{\phi},X}\hat{S}_Y T 
                    + \inn{\hat{\nabla}\tilde{\phi},Y} \hat{S}_X T \right),
	\end{split}
	\end{align}
	 %%%%%%%%%%%%%%%%%%%%%%%%%%%%%%
	
	\begin{align}\label{GW-HVHh}		%HVH h
	\begin{split}
		e^{-2\tilde{\phi}} \tilde{R}^\h_M(X,T)Y
                    =\; & \hat{R}^\h_M(X,T)Y  - \inn{\hat{\nabla}\tilde{\phi},Y} \hat{A}^*_X T\\
                    & - 2\inn{\hat{\nabla}\tilde{\phi},X} \hat{A}^*_Y T
                      + \inn{\hat{A}_X Y,T} \hat{\nabla}\tilde{\phi},
	\end{split}
	\end{align}

	 %%%%%%%%%%%%%%%%%%%%%%%%%%%%%%
	\begin{equation}\label{GW-VVHh}	%VVH h
		e^{-2\tilde{\phi}} \tilde{R}^\h_M(T_1,T_2)X = \hat{R}^\h_M(T_1,T_2)X + (1-e^{2\tilde{\phi}}) 
											(\hat{A}^*_{\hat{A}^*_X T_1}T_2 - \hat{A}^*_{\hat{A}^*_X T_2}T_1), 
	\end{equation}
	
	\begin{equation}\label{GW-VVVh}	%VVV h
	\tilde{R}^\h_M(T_1,T_2)T_3 = \hat{R}^\h_M(T_1,T_2)T_3, \;\; \text{and}
	\end{equation}
  %%%%%%%%%%%%%%%%%%%%%%%%%%%%%%

	\begin{align}\label{GW-VVVv}		%VVV v
	\begin{split}
	\tilde{R}^\v_M(T_1, T_2) T_3 =\; &
	\hat{R}^\v_M(T_1, T_2) T_3 
	+ (1-e^{2\tilde{\phi}}) \left\{ \hat{S}_{\hat{\sigma}(T_2, T_3)} T_1 - \hat{S}_{\hat{\sigma}(T_1, T_3)} T_2 \right\} \\
	& +e^{2\tilde{\phi}} \left\{\left(\hat{S}_{\hat{\nabla} \tilde{\phi}}-|\hat{\nabla} \tilde{\phi}|^2 I\right) (\inn{T_2, T_3} T_1- \inn{T_1, T_3} T_2)\right. \\
	&  + \left. \inn{\hat{\nabla} \tilde{\phi}, \hat{\sigma}(T_2, T_3)}T_1 - \inn{\hat{\nabla} \tilde{\phi}, \hat{\sigma}(T_1, T_3)} T_2\right\},
	\end{split}
	\end{align}
     %%%%%%%%%%%%%%%%%%%%%%%%%%%%%%
    where $I$ denotes the identity  map, and, as in \cite{gromoll_walschap}, in \eqref{GW-HHHh}, we are conflating $\hat{R}_B(X,Y)Z$ with its horizontal lift. 
\end{prop}

 As mentioned above, the functions $\tilde{\omega}_h$ and $\tilde{\omega}_v$ that define our deformation will be chosen to be $C^1$-small. This together with Propositions \ref{things-are-small-lemma-1}, \ref{things-are-small-lemma-2}, and \ref{gromoll-walschap-prop} and Corollary \ref{no biggie} imply that most components of the difference of the curvature operators
\[ \frak{D} \cal R_M := \widetilde{\cal R}_M - \cal R_M \]
are small. We make this precise in the following proposition. 
 
\begin{prop}\label{only-two-large-curvatures-prop}
Let $M$ be a compact smooth manifold, and let
\[\pi:(M^{\red{n}},g_M)\to(B^b,g_B)\]
be a Riemannian submersion. Then for every $\epsilon >0$ there is $\delta>0$ that also depends on $g_M$ and has the following properties.

Let $\tilde{g}_M$ be as in (\ref{warped M metric defn}).
Let $\cal R_M$ and $\widetilde{\cal R}_M$ be the curvature operators of $g_M$ and $\tilde{g}_M$, respectively, and set
\[ \frak{D} \cal R_M := \widetilde{\cal R}_M - \cal R_M. \] 
If 
\begin{equation}\label{C1 small}
    \norm{{\omega}_h}_{C^1} < \delta, \ \norm{\tilde{\omega}_v}_{C^1} < \delta,
\end{equation}
\begin{eqnarray}
    - 2 \hess^M_{{\omega}_h}(\cdot,\cdot) &\geq& ( \lambda_{h,h} +1) g_B(\cdot, \cdot),  \; \; \text{and} \nonumber\\
    -\hess^M_{\tilde{\omega}_v}(\cdot, \cdot)|_{\cal H\oplus \cal H} &\geq& (\lambda_{h,v} +1)  g_M(\cdot, \cdot)|_{\cal H\oplus \cal H},
\end{eqnarray}
then for all horizontal $g_M$-orthonormal vectors $Y,Z$ and vertical $g_M$-orthonormal vectors $U,V$, 
\begin{eqnarray}\label{delta R Y Z}
        \norm{ \frak{D} \cal R_M(U\wedge V) } &<&  \epsilon, \\
        %%%%%%%%%%%%%%%%%%%%%%%%%%
    \norm{ \frak{D} \cal R_M(Y\wedge Z)^{( \cal H \wedge \cal V ) \oplus ( \cal V \wedge \cal V)} }& < & \epsilon,    \label{delta R Y Z-2}\\
%%%%%%%%%%%%%
\label{delta R Y T}     
    \norm{ \frak{D} \cal R_M(Y\wedge V)^{( \cal H \wedge \cal H ) \oplus ( \cal V \wedge \cal V)} }  &< & \epsilon,\\
%%%%%%%%%%%%%%%%%%%%%%%%%%%%%%%%%%%%%%%%%
g(\frak{D}\cal{R}_M\left( Y\wedge V\right) ,Y\wedge V) &\geq &\lambda
_{h,v},   \; \; \text{and} \label{just better than--XV-2}       \\
%%%%%%%%%%%%%%%%%%%%%%%%%%%%%%%%%%%%%%%
  g(\frak{D}\cal{R}_M\left( X\wedge Y\right) ,X\wedge Y) &\geq &\lambda
_{h,h}.  \label{just better than--XY} 
%%%%%%%%%%%%%%%%%%%%%%%%%%%%                                              
\end{eqnarray}
\end{prop}
\begin{proof}
Since $M$ is compact, each of the fundamental tensors, $A, A^*, \sigma$, and $S$, for the submersion $ \pi:(M,g_M)\to(B^b,g_B)$ 
is uniformly bounded. This together with Proposition \ref{things-are-small-lemma-2} and the hypothesis that $\norm{\tilde{\omega}_h}_{C^1} < \delta$ gives us that each of these tensors is within $O(\delta)$ of the corresponding
tensors $\hat{A}$, $\hat{A}^*$, $\hat{S}$, and $\hat{\sigma}$ for $(M, \hat{g}_M)\to (B,\hat{g}_B)$. 
     This together with Corollary \ref{no biggie} gives us that, with one exception, in each of the Equations (\ref{GW-HHHv})--(\ref{GW-VVVv}), the difference between the two curvature tensors satisfies
\begin{equation}\label{small Delta R}
\norm{\tilde{R} - \hat{R}} < O(\delta).
\end{equation}
The exception is Inequality (\ref{GW-HVHv}), which after substituting $Y$ for $X$, $V$ for $T$, and $\tilde{\omega}_v$ for $\tilde{\phi}$ becomes
\begin{eqnarray*}
    \tilde{R}^\v_M(V,Y)Y & =  & \hat{R}^\v_M(V,Y)Y - {\widehat{\hess}}^M_{\tilde{\omega}_v}(Y,Y)V + O(\delta),
\end{eqnarray*}
which combined with Inequality \eqref{change-in-hessian-is-small} and Corollary \ref{no biggie} gives us
     \begin{eqnarray}\label{Hessian exception}
 \|  \tilde{R}^\v_M(V,Y)Y - \left( R^\v_M(V,Y)Y - {\hess}^M_{\tilde{\omega}_v}(Y,Y)V \right) \| <\epsilon,
     \end{eqnarray}
provided $\delta$  is sufficiently small. 

Together \eqref{small Delta R} and \eqref{Hessian exception} imply that \eqref{delta R Y Z}, \eqref{delta R Y Z-2}, and \eqref{delta R Y T} hold, provided $\delta$  is sufficiently small.  Since we have also assumed that 
 $$
 - \hess^M_{\tilde{\omega}_v}(\cdot,\cdot)|_{\cal H\oplus \cal H}\geq (\lambda_{h,v}+1)g_{M}(\cdot,\cdot)|_{\cal H\oplus \cal H}, 
 $$\eqref{Hessian exception} also implies that 
\begin{equation*}
 g(\frak{D}\cal{R}_M\left( Y\wedge V\right) ,Y\wedge V) \geq  \lambda_{h,v}, \; \; \text{provided $\varepsilon\in \left(0,\tfrac{1}{2}\right),$}    
\end{equation*}
as claimed in  \eqref{just better than--XV-2}. 

To obtain \eqref{just better than--XY}, we note that  Equation (\ref{GW-HHHh}) together with the fact that 
$\| \tilde{\omega}_v\|_{C^1} < \delta$ gives us
\begin{eqnarray}\label{GW-HHHh--2}	
		\tilde{R}^\h_M(X,Y)Z &=& e^{2\tilde{\omega}_v}\hat{R}^\h_M(X,Y)Z 
                             + (1-e^{2\tilde{\omega}_v})\tilde{R}_B(X,Y)Z \nonumber\\
                             &=& \hat{R}^\h_M(X,Y)Z + O(\delta).
	\end{eqnarray}
    
Equation (\ref{no-change-in-A}) together with the Gray--O'Neill  Equation and the fact that the restrictions of $\hat{g}_M$ and $g_{M}$ to the vertical distributions coincide gives us that
%%%%%%%%%%%%%%%%%
\begin{eqnarray}\label{Grey-O'N}	\label{H Delta}
\norm{(\hat{R}^\h_M -R_{M}^\h)(X,Y)Z} &=& \norm{(\hat{R}_B -R_B)(X,Y)Z} \nonumber \\
                            &=& \norm{\frak{D}R_B(X,Y)Z}, 
\end{eqnarray}
by the definition of $\frak{D}R_B$. Together (\ref{H Delta}) and (\ref{Delta R at p}) imply that  
\begin{equation*}
    \norm{\hat{R}^\h_M -R_{M}^\h + (2\hess_{\omega_h} \circ g_B)}  < \frac{\epsilon}{4}, \; \; \text{provided $\delta$ is  small enough,}
\end{equation*}
which combined with \eqref{GW-HHHh--2} gives us that 
\begin{equation*}
    \norm{\tilde{R}^\h_M -R_{M}^\h + (2\hess_{\omega_h} \circ g_B)}  < \frac{\epsilon}{2},
\end{equation*}
provided $\delta$ is sufficiently small. 

The previous display together with our hypothesis that $- 2 \hess^M_{{\omega}_h}(\cdot,\cdot) \geq ( \lambda_{h,h} +1) g_B(\cdot, \cdot),$ gives us
\begin{equation}\label{Delta R horiz}
    g(\frak{D}\cal{R}_M\left( X\wedge Y\right) ,X\wedge Y) \geq \lambda
_{h,h},  
\end{equation}
as claimed in \eqref{just better than--XY}.
\end{proof}

To finish the proof of Theorem \ref{trying too hard}, we fix $p\in B$ and remind the reader that (\ref{warped M metric defn}) defines the  metric $\tilde{g}_M$ that eventually proves Theorem \ref{trying too hard} as follows:
\begin{equation}\label{pet and repeat}
   \tilde{g}_M := e^{2\tilde{\omega}_h}g_M|_{\cal H\oplus \cal H} + e^{2\tilde{\omega}_v}g_M|_{\cal V \oplus \cal V}.
\end{equation}
Thus $\tilde{g}_M - g_M$ is determined by the choice of functions $\tilde{\omega}_h$ and $\tilde{\omega}_v$, which in turn we will construct via Lemma \ref{existence lemma omega} and Corollary \ref{existence lemma omega vertical}. The functions $\tilde{\omega}_h$ and $\tilde{\omega}_v$  are each constrained by a choice of the four constants $\tau,\eta,\epsilon,C$ mentioned in Lemma \ref{existence lemma omega} and Corollary \ref{existence lemma omega vertical}. Thus to prove Theorem \ref{trying too hard} we show that there are two quadruples of constants 
$$
(\tau_{h},\eta_{h},\epsilon _{h},C_{h}) \;\; \text{and} \; \;  (\tau _{v},\eta _{v},\epsilon _{v},C_{v})
$$
so that when we use the corresponding functions $\tilde{\omega}_h$ and $\tilde{\omega}_v$ to define $\tilde{g}_M$ via (\ref{pet and repeat}),  the resulting metric satisfies the conclusion of Theorem \ref{trying too hard}. We  explain this more formally via the following two definitions. 
%%%%%%%%%%%%%%%%%%%
\begin{defn}
Let $\tilde{\epsilon} >0$. A quadruple $(\tau_{h},\eta_{h},\epsilon_{h},C_{h})$ is called $\boldsymbol{\tilde{\epsilon}}$\textbf{-Horizontal Admissible} if there exists a corresponding function $\omega _{h}:B \to \R$ constructed via Lemma \ref{existence lemma omega} so that $(B,\tilde{g}_{B})$ satisfies the conclusion of Proposition \ref{neg Ricci at p}, with $\tilde{\epsilon}$ playing the role of $\epsilon$.
\end{defn}

\begin{defn}
Given $\epsilon, \tilde{\epsilon}>0$, let $(\tau _{h},\eta _{h},\epsilon _{h},C_{h})$ be an $\tilde{\epsilon} $-Horizontal Admissible quadruple, and let $\omega _{h}:B \to \R$ be a corresponding function, constructed via Lemma \ref{existence lemma omega}.

A quadruple $(\tau _{v},\eta _{v},\epsilon _{v},C_{v})$ is called  $\boldsymbol{(\epsilon,\tau_{h},\eta _{h},\epsilon _{h},C_{h})}$\textbf{-Vertical Admissible} if there exists a corresponding function $\tilde{\omega}_{v}:B\to \mathbb{R}$ constructed via Corollary \ref{existence lemma omega vertical} so that $(M,\tilde{g}_{M})$,
defined via (\ref{warped M metric defn}), satisfies 
\begin{equation*}
\tric_{b}^{M} \geq \min \ric_{b}^{M} -\epsilon.
\end{equation*}
\end{defn}

In particular, Theorem \ref{trying too hard} follows from the following proposition.

\begin{prop}
Given $\epsilon>0,$ there are an $\eta _{h}^{0}, \tilde{\epsilon}_{0}>0
$ with the following property: For all $\eta _{h}\in \left( 0,\eta
_{h}^{0}\right) $ and $\tilde{\epsilon} \in (0, \tilde{\epsilon}_{0})$, there is an  $
\tilde{\epsilon}  $-Horizontal Admissible $(\tau _{h},$ $\eta _{h},\epsilon
_{h},C_{h})$ and an $(\epsilon,\tau _{h},\eta _{h},\epsilon
_{h},C_{h})$-Vertical Admissible quadruple $(\tau _{v},$ $\eta _{v},\epsilon
_{v},C_{v})$.
\end{prop}

\begin{proof}
From Proposition \ref{neg Ricci at p} we know that there is an $\tilde{\epsilon}_{0}>0$ so that
for all $\tilde{\epsilon} \in (0,\tilde{\epsilon}_{0})$ there are $\tilde{\epsilon} $
-Horizontal Admissible quadruples, $(\tau _{h},$ $\eta _{h},\epsilon
_{h},C_{h}).$  In particular, the constant $C_h$ is positive, so by Inequality (\ref{Hessian inequal +}),  for all unit vectors  $Z \in TB$,
    %%%%%%%%%%
    \begin{equation} \label{Hessian inequal ++}
     - \epsilon_h \leq \hess_{\omega_h}(Z,Z) \leq 3C_h.
    \end{equation}
On the other hand, via (\ref{gen Hess omega v}) we have that the Hessian of $\tilde{\omega}_v$  satisfies
    %%%%%%%%%%
    \begin{equation} \label{Hessian inequal --}
   3C_v \leq \hess_{\tilde{\omega}_v}(Z,Z) \leq \epsilon_v
    \end{equation}
for all unit vectors $Z \in TM$. Additionally, we have from  (\ref{spec Hess omega v}) that  for all unit horizontal vectors  $Z \in TM|_{B(\pi^{-1}(p),\tau_v)},$
     \begin{equation} \label{Hessian inequal -b-}
   \hess_{\tilde{\omega}_v}(Z,Z) \leq C_v <0.
    \end{equation}
We set 
%%%%%%%%%%%%%%%%%%%%%%%%%%%%%%%%%%%%%%%%%%
\begin{eqnarray}\label{nearby lambdas}
\lambda_{h,h} : = -6 C_h  -1 \; \;  \text{and} \; \;  \lambda_{h,v} :=  -C_v -1,
\end{eqnarray}
and we further require that 
%%%%%%%%%%%%%%%%%%%%%  
\begin{eqnarray}
3\eta _{h} &<&\tau _{v} \; \; \; \text{and}  \label{supp omega h} \\
(b-1)\lambda_{h,h} + \lambda_{h,v} &>& 0. \label{C in equal}
\end{eqnarray}

In particular, (\ref{supp omega h}) implies that  the support of $\tilde{\omega}_h$ is contained in $B(\pi^{-1} (p) ,\tau_v)$. Thus, we have that (\ref{Hessian inequal -b-}) holds throughout the support of $\tilde{\omega}_h$. Combining this with \eqref{Hessian inequal ++} and \eqref{nearby lambdas}  we see that on $B(p,\tau_v)$ we have that 
\begin{eqnarray*}
    - 2 \hess^M_{\omega_h}(\cdot,\cdot) \geq (\lambda_{h,h} +1) g_B (\cdot,\cdot),
 \end{eqnarray*}   
 and on $B(\pi^{-1}(p),\tau_v)$  we have that 
    \begin{eqnarray*}
    \; \; -\hess^M_{\tilde{\omega}_v}(\cdot,\cdot)|_{\cal H\oplus \cal H} \geq (\lambda_{h,v} +1) g_M (\cdot,\cdot)|_{\cal H\oplus \cal H},
\end{eqnarray*}
as in the hypotheses of Proposition \ref{only-two-large-curvatures-prop}.

Therefore by Proposition \ref{only-two-large-curvatures-prop}, $\frak{D}\cal{R}_M$ satisfies \eqref{abstract condition for closeness}--\eqref{just better than--XV}   , provided $\epsilon_h$ and $\epsilon_v$ are small enough. This together with the constraint \eqref{C in equal} gives that on $B(\pi^{-1}(p),\tau_v)$, $\frak{D}\cal{R}_M$ satisfies the (\ref{RW Sum condition})--version of Lemma \ref{verification cor}. So  
\[
\tric_{b}^{M} \geq \min \ric_{b}^{M} -\epsilon \quad \text{on $B(\pi^{-1}(p),\tau_v)$.}
\]

On the other hand, on $M\setminus B(\pi^{-1} (p),\tau_v)$, we see that $\tilde{\omega}_h \equiv 0$, so we still have that the (\ref{RW Sum condition}) version of  Lemma \ref{verification cor} holds wherever $\hess_{\tilde{\omega}_v}<0$. 

Finally, notice that the upper bound in (\ref{Hessian inequal --}) gives us that the (\ref{all small}) version of Lemma \ref{verification cor} holds wherever $\hess_{\tilde{\omega}_v}\geq 0$. We therefore have 
\[\tric_{b}^{M} \geq \min \ric_{b}^{M} -\epsilon\]
throughout $M$, as claimed.

\end{proof}

{\color{red}
To prove Corollary \ref{grp action cor}, we let $G\times M \longrightarrow M$ be free and isometric with respect to $g_M$, and we set $B: = M/G$. 
Observe that the functions from \eqref{omega h omega v}, 
\begin{equation*}
    \tilde{\omega}_h  :=  \omega_h \circ\pi \quad\text{and}\quad \tilde{\omega}_v := \omega_v  \circ\pi,
\end{equation*}
that we used  define $\tilde{g}_M$ are both $G$--invariant. This together with the fact that the horizontal and vertical distributions of  $\pi$ are $G$-invariant implies that metric, 
\begin{equation*}
   \tilde{g}_M := e^{2\tilde{\omega}_h}g_M|_{\cal H\oplus \cal H} + e^{2\tilde{\omega}_v}g_M|_{\cal V \oplus \cal V},
\end{equation*}
is also $G$-invariant, as claimed in the conclusion of Corollary \ref{grp action cor}.}

%%%%%%%%%%%%%%%%%%%%%%%%%%%%%%%%%%%%%%%%%%%%%%%%%%%
%%%%%%%%%%%%%%%%%%%%%%%%%%%%%%%%%%%%%%%%%%%%%%%%%%%
%%%%%%%%%%%%%%%%%%%%%%%%%%%%%%%%%%%%%%%%%%%%%%%%%%%

\bibliographystyle{plain} % We choose the "plain" reference style
\bibliography{bibliography} % Entries are in the tot-geo.bib file

\end{document}